\documentclass{article}

\PassOptionsToPackage{numbers, compress}{natbib}

\usepackage[final]{neurips_2021}

\usepackage[utf8]{inputenc} 
\usepackage[T1]{fontenc}    
\usepackage[colorlinks=true, allcolors=blue]{hyperref}       
\usepackage{url}            
\usepackage{booktabs}       
\usepackage{amsfonts}       
\usepackage{nicefrac}       
\usepackage{microtype}      
\usepackage{xcolor}         
\usepackage{graphicx}
\usepackage{subfigure}
\usepackage{dsfont}

\title{A Continuized View on Nesterov Acceleration for Stochastic Gradient Descent and Randomized Gossip}

\author{
	Mathieu Even\textsuperscript{1,$\star$},
	Raphaël Berthier\textsuperscript{1,$\star$},
	Francis Bach\textsuperscript{1},
	Nicolas Flammarion\textsuperscript{2},\\
	\textbf{Pierre Gaillard\textsuperscript{3},
	Hadrien Hendrikx\textsuperscript{1},
	Laurent Massoulié\textsuperscript{1,4} and
	Adrien Taylor\textsuperscript{1}} \\
\\
	\textsuperscript{$\star$} Equal contributions\\
	\\
	\textsuperscript{1}Inria - Département d’informatique de l’ENS, PSL Research University, Paris, France\\
\\
	\textsuperscript{2}School of Computer and Communication Sciences \\
	Ecole Polytechnique F\'ed\'erale de Lausanne\\
	\\
		\textsuperscript{3}Univ. Grenoble Alpes, Inria, CNRS, Grenoble INP, LJK, 38000 Grenoble, France\\
		\\
	\textsuperscript{4}MSR-Inria Joint Centre \\
}

\renewcommand{\leq}{\leqslant}
\renewcommand{\geq}{\geqslant}
\renewcommand{\le}{\leqslant}
\renewcommand{\ge}{\geqslant}

\def\<{\langle}
\def\>{\rangle}
\def\|{\Vert}

\definecolor{NavyBlue}{rgb}{0.1,0.1,0.6}

\def\bfone{{\boldsymbol 1}}

\newcommand{\diff}{\mathrm{d}}

\newcommand{\N}{\mathbb{N}}

\newcommand{\E}{\mathbb{E}}
\newcommand{\R}{\mathbb{R}}

\renewcommand{\P}{\mathds{P}}

\newcommand{\cE}{\mathcal{E}}
\newcommand{\cP}{\mathcal{P}}

\newcommand{\cG}{\mathcal{G}}
\newcommand{\cF}{\mathcal{F}}

\newcommand{\cV}{\mathcal{V}}

\newcommand{\cL}{\mathcal{L}}

\newcommand{\esp}[1]{{{\E\left[ #1\right]}}} 
\newcommand{\NRM}[1]{{{\left\| #1\right\|}}} 

\newcommand{\lambdaz}{\lambda^{(z)}}
\newcommand{\lambdax}{\lambda^{(y)}}
\newcommand{\edgeuv}{{\{v, w\}}}
\newcommand{\edgeuvk}{{\{v_k, w_k\}}}

\newcommand{\proj}{R}

\usepackage{amsfonts}
\usepackage{amsthm}
\usepackage{amsmath}
\usepackage{amssymb}
\usepackage{mathrsfs}
\usepackage{amscd}
\usepackage{caption}
\usepackage{indentfirst}
\usepackage{cleveref}

\DeclareMathOperator{\Id}{Id}

\newtheorem{definition}{Definition}
\newtheorem{theorem}{Theorem}
\newtheorem{proposition}{Proposition}
\newtheorem{cor}{Corollary}

\newtheorem{remark}{Remark}

\begin{document}

\maketitle

\begin{abstract}
We introduce the ``continuized'' Nesterov acceleration, a close variant of Nesterov acceleration whose variables are indexed by a continuous time parameter. The two variables continuously mix following a linear ordinary differential equation and take gradient steps at random times. This continuized variant benefits from the best of the continuous and the discrete frameworks: as a continuous process, one can use differential calculus to analyze convergence and obtain analytical expressions for the parameters; and a discretization of the continuized process can be computed exactly with convergence rates similar to those of Nesterov original acceleration. We show that the discretization has the same structure as Nesterov acceleration, but with random parameters. 
We provide continuized Nesterov acceleration under deterministic as well as stochastic gradients, with either additive or multiplicative noise. 
Finally, using our continuized framework and expressing the gossip averaging problem as the stochastic minimization of a certain energy function, we provide the first rigorous acceleration of asynchronous gossip algorithms.
\end{abstract}

\section{Introduction}
\label{sec:introduction}
In the last decades, the emergence of numerous applications in statistics, machine learning and signal processing has led to a renewed interest in first-order optimization methods \citep{bottou2018optimization}. They enjoy a low iteration cost necessary to the analysis of large datasets. The performance of first-order methods was largely improved thanks to acceleration techniques (see the review by~\citet{daspremont2021acceleration} and the many references therein), starting with the seminal work of \citet{nesterov1983method}.

Let $f:\R^d \to \R$ be a convex and differentiable function, minimized at $x_* \in \R^d$. We assume throughout the paper that $f$ is $L$-smooth, i.e.,
\begin{align} \label{eq:smoothness}
	\forall x,y \in \R^d, \qquad f(y) \leq f(x) + \langle \nabla f(x), y-x \rangle + \frac{L}{2} \Vert y-x \Vert^2 \, .
\end{align}
In addition, we sometimes assume that $f$ is $\mu$-strongly convex for some $\mu > 0$, i.e., 
\begin{align} \label{eq:strong_convexity}
	\forall x,y \in \R^d, \qquad f(y) \geq f(x) + \langle \nabla f(x), y-x \rangle + \frac{\mu}{2} \Vert y-x \Vert^2 \, .
\end{align}
For the problem of minimizing $f$, gradient descent is well-known to achieve a rate $f(x_k)-f(x_*) = O(k^{-1})$ in the smooth case, and a rate $f(x_k) - f(x_*) = O((1-\mu/L)^k)$ in the smooth and strongly convex case. In both cases, Nesterov introduced an alternative method with essentially the same iteration cost, while achieving faster rates: it converges with rate $O(k^{-2})$ in the smooth convex case and with rate $O((1-\sqrt{\mu/L})^k)$ in the smooth and strongly convex case~\citep{nesterov2003introductory}. These rates are then optimal among all black-box first-order methods that access gradients and linearly combine them~\citep{nesterov2003introductory,nemirovskij1983problem}.

Nesterov acceleration relies on several sequences of iterates---two or three, depending on the formulation---and on a clever blend of gradient steps and mixing steps between the sequences. 
Different interpretations and motivations underlying the precise structure of accelerated schemes were approached in many works, including~\citep{bubeck2015geometric,flammarion2015averaging, arjevani2015lower,kim2016optimized,allen2014linear}.
A large number of these works studied continuous time equivalents of Nesterov acceleration, obtained by taking the limit when stepsizes vanish, or from a variational framework. The continuous time index~$t$ of the limit allowed to use differential calculus to study the convergence of these equivalents. Examples of studies relying on continuous time interpretation include~\citep{su2014differential,krichene2015accelerated,wilson2016lyapunov,wibisono2016variational,betancourt2018symplectic,diakonikolas2019approximate,shi2018understanding,shi2019acceleration,attouch2018fast,attouch2019rate,zhang2018direct,muehlebach2019dynamical}.

\paragraph{Continuized Nesterov acceleration.} In this paper, we propose another continuous time equivalent to Nesterov acceleration, which we refer to as the \emph{continuized} Nesterov acceleration, which avoids vanishing stepsizes. It is built by considering two sequences $x_t, z_t \in \R^d$, $t \in \R_{\geq 0}$, that continuously mix following a linear ordinary differential equation (ODE), and that take gradient steps at random times $T_1, T_2, T_3, \dots$. Thus, in this modeling, mixing and gradient steps alternate randomly. 

Thanks to the continuous index $t$ and some stochastic calculus, one can differentiate averaged quantities (expectations) with respect to~$t$. In particular, this leads to simple analytical expressions for the optimal parameters as functions of $t$, while the optimal parameters of Nesterov accelerations are defined by recurrence relations that are complicated to solve. 

The discretization $\tilde{x}_k = x_{T_k}, \tilde{z}_k = z_{T_k}, k\in \N$, of the continuized process can be computed directly and exactly: the result is a recursion of the same form as Nesterov iteration, but with randomized parameters, and performs similarly to Nesterov original deterministic version both in theory and in simulations. 


The continuized framework can be adapted to various settings and extensions of Nesterov acceleration. In what follows, we study how the continuized acceleration behaves in the presence of \emph{additive} and \emph{multiplicative} noise in the gradients. In the multiplicative noise setting, our acceleration satisfies a convergence rate similar to that of \cite{jain2018accelerating} and depends on the \emph{statistical condition number} of the problem at hand. The two acceleration schemes are not directly comparable as we work in a continuized setting and only deal with pure multiplicative noise. Our analysis is nevertheless much simpler, as it closely mimics that of Nesterov acceleration.

\paragraph{Application to accelerated gossip algorithms.} 
The continuized modeling is natural in asynchronous parallel computing where gradient steps arrive at random times. More importantly, there are situations where the continuized version of Nesterov acceleration can be practically implemented while the original acceleration can not. In distributed settings, for instance, the total number $k$ of gradient steps taken in the network is typically not known to each particular node; an advantage of the continuized acceleration is that it requires to know only the time $t$ and not $k$.

Gossip algorithms typically feature such asynchronous and distributed behaviors \citep{boyd2006randomized}. In gossip problems, nodes of a network aim at computing the global average of all their values by communicating only locally (with their neighbors), and without centralized coordination. In this set-up, pairs of adjacent nodes communicate at random times, asynchronously, and in parallel, so that the total number of past communications in the network at a given time is unknown to all nodes. In this paper, we formulate the gossip problem as a stochastic optimization problem. Thanks to the continuized formalism, we  naturally obtain accelerated gossip algorithms that can be implemented in an asynchronous and distributed fashion.

Synchronous gossip algorithms rely on all nodes to communicate simultaneously \citep{dimakis2010synchgossip}. Accelerating synchronous gossip algorithms have been studied in previous works, including \emph{SSDA} \citep{scaman2017optimal}, Chebyshev acceleration \citep{montijano2011chebgossip}, Jacobi-Polynomial acceleration \citep{berthier2020accelerated}.
To that day, acceleration in the asynchronous setting has also been studied in a few works (see for instance geographic gossip \citep{Dimakis_2008}, shift registers \citep{LIU2013873}, \emph{ESDAC} \citep{hendrikx2018accelerated}, and randomized Kaczmarz methods \cite{loizou2019provably}). 
However, no algorithm in an asynchronous framework has been rigorously proven to achieve an accelerated rate for general graphs \citep{dimakis2008geographic}. Other acceleration schemes \citep{hendrikx2018accelerated,loizou2019provably} relied on additional synchronizations between nodes, such as the knowledge of a global iteration counter. This departs from purely asynchronous operations, hence causing practical limitation. 
Our accelerated randomized gossip algorithm (Section \ref{sec:gossip}) recovers the same accelerated rates, and only requires the knowledge of a common continuous-time $t\in\R_{\geq0}$.

In this context, the continuized acceleration should be seen as a close approximation to Nesterov acceleration, that features both an insightful and convenient expression as a continuous time process and a direct implementation as a discrete iteration. We thus hope to contribute to the understanding of Nesterov acceleration. In practice, the continuized framework is relevant for handling asynchrony in decentralized optimization, where agents of a network can not share a global iteration counter, preventing accelerated decentralized and asynchronous methods.

\paragraph{Notations.}
The index $k$ always denotes a non-negative integer, while indices $t,s$ always denote non-negative reals.

\paragraph{Structure of the paper.} In Section \ref{sec:reminders}, we recall standard results on gradient descent and Nesterov acceleration. In Section \ref{sec:continuized}, we introduce a continuized variant of Nesterov acceleration. In Section~\ref{sec:implementation}, we show that discretizing the continuized acceleration yields an iterative method similar to that of Nesterov but with random parameters. In Section \ref{sec:stochastic}, we study continuized Nesterov acceleration under pure-multiplicative noise. We finally present accelerated asynchronous algorithms for the gossip problem in Section \ref{sec:gossip}, as well as for decentralized optimization in Section \ref{sec:decentralized_opt}.

\section{Reminders on Nesterov acceleration}
\label{sec:reminders}

For the sake of comparison, let us first recall the classical Nesterov acceleration. To improve the convergence rate of gradient descent, Nesterov introduced iterations of three sequences, parametrized by $\tau_k, \tau_k',\gamma_k, \gamma_k',  k\geq 0$, of the form
\begin{align}
&y_k = x_k + \tau_k(z_k-x_k) \, ,  \label{eq:nest-1}\\
&x_{k+1} = y_k - \gamma_k \nabla f(y_k) \, , \\
&z_{k+1} = z_k + \tau_k'(y_k-z_k) - \gamma_k' \nabla f(y_k ) \label{eq:nest-3}\, . 
\end{align}
Depending on whether the function $f$ is known to be (\ref{it:nest-cvx}) convex, or (\ref{it:nest-str-cvx}) strongly convex with a known strong convexity parameter, Nesterov provided a set of parameter choices for achieving acceleration.

\begin{theorem}[Convergence of accelerated gradient descent] Nesterov accelerated scheme satisfies:
	\label{thm:nesterov}
	\begin{enumerate}
		\item\label{it:nest-cvx} Choose the parameters $\tau_k = 1 - \frac{A_k}{A_{k+1}}, \tau_k' = 0, \gamma_k = \frac{1}{L}, \gamma_k' = \frac{A_{k+1}-A_k}{L}, k\geq 0$, where the sequence $A_k, k\geq 0$, is defined by the recurrence relation
		\begin{align*}
		&A_0 = 0 \, , &&A_{k+1} = A_k + \frac{1}{2}(1+\sqrt{4 A_k + 1}) \, .
		\end{align*}
		Then
		\begin{align*}
		f(x_k) - f(x_*) \leq \frac{2 L \Vert x_0 - x_* \Vert^2}{k^2} \, .
		\end{align*}
		\item\label{it:nest-str-cvx} Assume further that $f$ is $\mu$-strongly convex, $\mu > 0$. Choose the constant parameters \\$\tau_k \equiv \frac{\sqrt{\mu/L}}{1+\sqrt{\mu/L}}$, $\tau_k'\equiv \sqrt{\frac{\mu}{L}}$, $\gamma_k \equiv \frac{1}{L}$, $\gamma_k'\equiv \frac{1}{\sqrt{\mu L}}$, $k\geq 0$. Then	  
		\begin{align*}
		f(x_k) - f(x_*) \leq \left(f(x_0) - f(x_*) + \frac{\mu}{2} \Vert z_0 - x_* \Vert^2 \right) \left(1 - \sqrt{\frac{\mu}{L}}\right)^k \, .
		\end{align*}
	\end{enumerate}
\end{theorem}
This result can be found as is in \citet[Sections 4.4.1 and 4.5.3]{daspremont2021acceleration}. In the convex case, Nesterov acceleration achieves the rate $O(1/k^2)$, whereas gradient descent achieves a rate $O(1/k)$ (see \cite[Corollary 2.1.2]{nesterov2003introductory} for instance). In the strongly convex case, Nesterov acceleration achieves the rate $O((1-\sqrt{\mu/L})^k)$, where gradient descent achieves a rate $O((1-\mu/L)^k)$ (see \citep[Theorem 2.1.15]{nesterov2003introductory} for instance). In both cases, this results in a significant speedup in practice, see Figure~\ref{fig:comparison}. 

From a high-level perspective, Nesterov acceleration iterates over several variables, alternating between gradient steps (always with respect to the gradient at $y_k$) and mixing steps, where the running value of a variable is replaced by a linear combination of the other variables. However, the precise way gradient and mixing steps are coupled is rather mysterious, and the success of the proof of Theorem \ref{thm:nesterov} relies heavily on the detailed structure of the iterations. In the next section, we try to gain perspective on this structure by developing a continuized version of the acceleration. 

\section{Continuized version of Nesterov acceleration} 
\label{sec:continuized}

This paper uses several mathematical notions related to random processes. The following sections expose the results from heuristic considerations of those notions, rigorously defined in
Appendix~\ref{ap:toolbox}. 

We argue that the accelerated iteration becomes more natural when considering two variables $x_t, z_t$ indexed by a continuous time $t \geq 0$, that are continuously mixing and that take gradient steps at random times. More precisely, let $T_1, T_2, T_3, \dots \geq 0$ be random times such that $T_1, T_2-T_1, T_3-T_2, \dots$ are independent identically distributed (i.i.d.), of law exponential with rate $1$ (any constant rate would do, we choose $1$ to make the comparison with discrete time $k$ straightforward). By convention, we choose that our stochastic processes $t \mapsto x_t, t \mapsto z_t$ are c\`adl\`ag almost surely, i.e., right continuous with well-defined left-limits $x_{t-}, z_{t-}$ (Definition \ref{def:cadlag} in Appendix~\ref{ap:toolbox}). Our dynamics are parametrized by functions $\gamma_t, \gamma'_t, \tau_t, \tau_t'$, $t \geq 0$. At random times $T_1, T_2, \dots$, our sequences take gradient steps 
\begin{align}
x_{T_k} = x_{T_k-} - \gamma_{T_k} \nabla f (x_{T_k-}) \, , \label{eq:gradient-1}\\
z_{T_k} = z_{T_k-} - \gamma_{T_k}' \nabla f (x_{T_k-}) \, . \label{eq:gradient-2}
\end{align}
Because of the memoryless property of the exponential distribution, in a infinitesimal time interval $[t, t+\diff t]$, the variables take gradients steps with probability $\diff t$, independently of the past. 
Between these random times, the variables mix through a linear, translation-invariant, ordinary differential equation (ODE)
\begin{align}
&\diff x_t = \eta_t (z_t - x_t) \diff t \, ,  \label{eq:continuous-part-1}\\
&\diff z_t = \eta_t'(x_t - z_t) \diff t \label{eq:continuous-part-2}\, .
\end{align}
Following the notation of stochastic calculus, we can write the process more compactly in terms of the Poisson point measure $\diff N(t) = \sum_{k\geq 1} \delta_{T_k}(\diff t)$, which has intensity the Lebesgue measure $\diff t$,
\begin{align}
\diff x_t = \eta_t (z_t - x_t) \diff t - \gamma_t \nabla f(x_t) \diff N(t) \, , \label{eq:continuized-1}\\
\diff z_t = \eta_t' (x_t - z_t) \diff t- \gamma_t' \nabla f(x_t) \diff N(t) \, .\label{eq:continuized-2}
\end{align}

Before giving convergence guarantees for such processes, let us digress quickly on why we can expect an iteration of this form to be mathematically appealing.

First, from a Markov chain indexed by a discrete time index $k$, one can associate the so-called \emph{continuized} Markov chain, indexed by a continuous time $t$, that makes transition with the same Markov kernel, but at random times, with independent exponential time intervals \citep{aldous1995reversible}. Following this terminology, we refer to our acceleration \eqref{eq:continuized-1}-\eqref{eq:continuized-2} as the continuized acceleration. The continuized Markov chain is appreciated for its continuous time parameter $t$, while keeping many properties of the original Markov chain; similarly the continuized acceleration is arguably simpler to analyze, while performing similarly to Nesterov acceleration. 

Second, it can also be compared with coordinate gradient descent methods, that are easier to analyze when coordinates are selected randomly rather than in an ordered way~\citep{wright2015coordinate}. Similarly, the continuized acceleration is simpler to analyze because the gradient steps \eqref{eq:gradient-1}-\eqref{eq:gradient-2} and the mixing steps \eqref{eq:continuous-part-1}-\eqref{eq:continuous-part-2} alternate randomly, due to the randomness of $T_k$, $k \geq 0$. 
\medskip

In analogy with Theorem \ref{thm:nesterov}, we give choices of parameters that lead to accelerated convergence rates, in the convex case \eqref{it:cvx} and in the strongly convex case \eqref{it:str-cvx}. Convergence is analyzed as a function of $t$. As $\diff N(t)$ is a Poisson point process with rate $1$, $t$ is the expected number of gradient steps done by the algorithm. Thus $t$ is analoguous to $k$ in Theorem \ref{thm:nesterov}. In the theorem below, $\E$ denotes the expectation with respect to the Poisson point process $\diff N(t)$, the only source of randomness.
\begin{theorem}[Convergence of continuized Nesterov acceleration]\label{thm:continuized} The continuized Nesterov acceleration satisfies the following two points.
\begin{enumerate}
		\item\label{it:cvx} Choose the parameters $\eta_t = \frac{2}{t}, \eta_t' = 0, \gamma_t = \frac{1}{L}, \gamma_t' = \frac{t}{2L}$. Then 
		\begin{align*}
			\E f(x_t) - f(x_*) \leq \frac{2L\Vert z_0 -x_* \Vert^2}{t^2} \, .
		\end{align*}
		\item\label{it:str-cvx} Assume further that $f$ is $\mu$-strongly convex, $\mu > 0$. Choose the constant parameters\\ $\eta_t = \eta_t' \equiv \sqrt{\frac{\mu}{L}}$, $\gamma_t \equiv \frac{1}{L}$, $\gamma_t' \equiv \frac{1}{\sqrt{\mu L }}$. Then 
				\begin{align*}
		\E f(x_t) - f(x_*) \leq \left(f(x_0) - f(x_*) + \frac{\mu}{2} \Vert z_0 - x_* \Vert^2\right) \exp \left(-\sqrt{\frac{\mu}{L}}t\right) \, .
		\end{align*}
	\end{enumerate}
\end{theorem}

We give an elementary sketch of proof in Appendix \ref{ap:sketch-prrof} and a complete proof in Appendix \ref{ap:proof-continuized}.
Many authors have proposed continuous-time versions of Nesterov acceleration using differential calculus, see the numerous references in the introduction. For instance, in~\citet{su2014differential}, an ODE is obtained from Nesterov acceleration by taking the joint asymptotic where the stepsizes vanish and the number of iterates is rescaled.
The resulting ODE must be discretized to be implemented;
choosing the right discretization is not straightforward as it introduces stability and approximation errors that must be controlled \citep{zhang2018direct,shi2019acceleration,sanz2020connections}. 

On the contrary, our continuous time process \eqref{eq:continuized-1}-\eqref{eq:continuized-2} does not correspond to a limit where the stepsizes vanish. However, in Appendix \ref{ap:scaling-limit}, we check that the random continuized acceleration has the same deterministic ODE scaling limit as Nesterov acceleration. This sanity check emphasizes that the continuized acceleration is fundamentally different from previous continuous-time equivalents.   

\begin{remark}
\label{rmk:geometric}
    A similar Markovian structure can be obtained in a discrete setting by flipping \emph{i.i.d.}~coins to trigger gradient steps. By denoting $p>0$ the probability to trigger a gradient step when flipping a coin, \emph{(i)} $p=1$ gives the classical setting, and \emph{(ii)} $p\to0$ while renormalizing time gives our continuized framework. In fact, this setting with updates triggered randomly is an interpolation between the classical and continuized frameworks, and consists in replacing exponential random variables by geometric random variables of parameter $p$ for the waiting-time between updates. We thus believe the convergence guarantees described here and in the following can be adapted for this discrete scheme.
\end{remark}


\section{Discrete implementation of the continuized acceleration with random parameters}
\label{sec:implementation}

In this section, we show that the continuized acceleration can be implemented exactly as a discrete algorithm. This contrasts with the discretization of ODEs that introduces discretization errors; here, we compute exactly
\begin{align*}
&\tilde{x}_k := x_{T_{k}} \, ,  &&\tilde{y}_k := x_{T_{k+1}-} \, , &&\tilde{z}_k := z_{T_{k}} \, ,
\end{align*}
with the convention that $T_0 = 0$. The three sequences $\tilde{x}_k, \tilde{y}_k, \tilde{z}_k$, $k \geq 0$, satisfy a recurrence relation of the same structure as Nesterov acceleration, but with random weights. The resulting randomized discrete algorithm satisfies performance guarantees similar to those of Nesterov acceleration. 

\begin{theorem}[Discrete version of continuized acceleration]
	\label{thm:discretization}
	For any stochastic process of the form \eqref{eq:continuized-1}-\eqref{eq:continuized-2}, we have
	\begin{align}
	&\tilde{y}_k = \tilde{x}_k + \tau_k(\tilde{z}_k-\tilde{x}_k) \, ,  \label{eq:discretization-1}\\
	&\tilde{x}_{k+1} = \tilde{y}_k - \tilde{\gamma}_k \nabla f(\tilde{y}_k) \, , \\
	&\tilde{z}_{k+1} = \tilde{z}_k + \tau_k'(\tilde{y}_k-\tilde{z}_k) - \tilde{\gamma}_k' \nabla f(\tilde{y}_k ) \, , \label{eq:discretization-3}
	\end{align}
	for some random parameters $\tau_k, \tau_k', \tilde{\gamma}_k, \tilde{\gamma}_k'$ (that are functions of $T_k, T_{k+1}, \eta_t, \eta_t', \gamma_t, \gamma_t'$).
	\begin{enumerate}
		\item For the parameters of Theorem \ref{thm:continuized}.(\ref{it:cvx}),
		$\tau_k = 1 - \left(\frac{T_k}{T_{k+1}}\right)^2$, $\tau_k' = 0$, $\tilde{\gamma}_k = \frac{1}{L}$, and $\tilde{\gamma}_k' = \frac{T_k}{2L}$. Then 
				\begin{align*}
			\E \left[T_k^2\left(f(\tilde{x}_{k}) - f(x_*)\right) \right] \leq {2L\Vert z_0 -x_* \Vert^2} \, .
		\end{align*}
		\item For the parameters of Theorem \ref{thm:continuized}.(\ref{it:str-cvx}), $\tau_k = \frac{1}{2}\left(1 - \exp\left(-2\sqrt{\frac{\mu}{L}}(T_{k+1}-T_k)\right)\right)$, \\$\tau_k' = \tanh\left(\sqrt{\frac{\mu}{L}}(T_{k+1}-T_k)\right)$, $\tilde{\gamma}_k = \frac{1}{L}$, and $\tilde{\gamma}_k' = \frac{1}{\sqrt{\mu L}}$. Then 
		\begin{align*}
			\E \left[\exp \left(\sqrt{\frac{\mu}{L}}T_k\right)\left(f(\tilde{x}_{k}) - f(x_*)\right) \right] \leq f(x_0) - f(x_*) + \frac{\mu}{2} \Vert z_0 - x_* \Vert^2  \, .
		\end{align*}	
	\end{enumerate}
\end{theorem} 
The law of $T_k$ is well known: it is the sum of $k$ i.i.d.~random variables of law exponential with rate~$1$; this is called an Erlang or Gamma distribution with shape parameter $k$ and rate $1$. One can use well-known properties of this law, such as its concentration around its expectation $\E T_k = k$, to derive corollaries of the bounds above. 
The performance guarantees are proved in Appendix \ref{ap:proof-continuized}, and the formula for the discretization is studied in \ref{ap:proof-thm-discretization}. In Appendix \ref{ap:simus-cont}, we provide simulations confirming that this discrete random algorithm has a performance similar to Nesterov's original acceleration. 

\section{Continuized Nesterov acceleration of stochastic gradient descent}
\label{sec:stochastic}

We now investigate the design of continuized accelerations of stochastic gradient descent. We assume that we do not have direct access to the gradient $\nabla f(x)$ but to a random estimate $\nabla f(x, \xi)$, where $\xi \in \Xi$ is random of law $\cP$. In the continuized framework, the randomness of the stochastic gradient and its time mix in a particularly convenient way. For similar reasons, Latz studied stochastic gradient descent as a gradient flow on a random function that is regenerated at a Poisson rate \cite{latz2021analysis}. However, this approach has the same shortcomings as the other approaches based on gradient flows: the subsequent discretization introduces non-trivial errors. We avoid this problem here. 

We keep the algorithms of the same form, replacing gradients by stochastic gradients. Let $\xi_1, \xi_2, \dots$ be i.i.d.~random variables of law $\cP$. We take stochastic gradient steps at the random times $T_1, T_2, \dots$, 
\begin{align*}
x_{T_k} = x_{T_k-} - \gamma_{T_k} \nabla f (x_{T_k-}, \xi_k) \, , \\
z_{T_k} = z_{T_k-} - \gamma_{T_k}' \nabla f (x_{T_k-}, \xi_k) \, . 
\end{align*}
Between these random times, the variables mix through the same ODE
\begin{align*}
&\diff x_t = \eta_t (z_t - x_t) \diff t \, ,  \\
&\diff z_t = \eta_t'(x_t - z_t) \diff t \, .
\end{align*}
This can be written more compactly in terms of the Poisson point measure $\diff N(t,\xi) = \sum_{k\geq 1} \delta_{(T_k,\xi_k)}(\diff t, \diff \xi)$ on $\R_{\geq 0} \times \Xi$, which has intensity $\diff t \otimes \cP$,
\begin{align}
\diff x_t = \eta_t (z_t - x_t) \diff t - \gamma_t \int_{\Xi}\nabla f(x_t, \xi) \diff N(t,\xi) \, , \label{eq:continuized-sgd-additive-1}\\
\diff z_t = \eta_t' (x_t - z_t) \diff t- \gamma_t'  \int_{\Xi}\nabla f(x_t, \xi) \diff N(t,\xi) \, . \label{eq:continuized-sgd-additive-2}
\end{align}
Here, the discussion depends on the properties satisfied by the stochastic gradients $\nabla f(x,\xi)$. In Appendix \ref{ap:additive}, we study the so-called \emph{additive noise} case. We show that the continuized acceleration satisfies perturbed convergence rates with the same choices of parameters as in Theorem \ref{thm:continuized}. We thus show some robustness of the above acceleration to additive noise. Instead, in this section, we focus on the so-called \emph{pure multiplicative noise} case, as it is crucial for the study of asynchronous gossip that follows. In this setting, parameters need to be chosen differently for our proof technique to work. A continuized acceleration is still possible, depending on the statistical condition number. 

We now focus on functions $f$ is of the following form, typical to least-squares supervised learning:
\begin{equation}
\label{eq:lst-sqr}
    \forall x \in \R^d,~f(x)=\E_{(a,b)\sim \cP} \left[\frac{1}{2}\left(b-\langle x,a\rangle \right)^2\right] \, ,
\end{equation}
where $\xi=(a,b)\in\R^d\times\R$ is random of law $\cP$. We assume that our \emph{stochastic first order oracle} is the gradient of one realization of the expectation, namely,
\begin{equation*}
    \nabla f(x,\xi) = -(b-\langle x,a \rangle)a \, , \qquad \xi=(a,b) \, .
\end{equation*}
We investigate \emph{noiseless}---or purely multiplicative---stochastic gradients, in the sense that almost surely, for $\xi=(a,b)\sim \cP$:
\begin{equation}\label{eq:noiseless}
    b=\langle x_{*},a \rangle,~\text{so that}~\nabla f(x_{*},\xi)=0 \, .
\end{equation}
Noiseless stochastic gradients are relevant in several situations, such as coordinate gradient descent with randomly sampled coordinates \cite{tseng2009coordinate,nesterov2012efficiency, wright2015coordinate} (where $\nabla f(x,\xi)=m\langle \nabla f(x),e_i\rangle e_i$ with $i$ uniformly random in $\{1, \dots, d\}$), over-parameterized regime for least squares regression \cite{vaswani2019fast}, function interpolation and gossip algorithms \cite{berthier2020tight}. 

For a symmetric non-negative matrix $A$ and a vector $x$, we denote $\NRM{x}_A^2=x^{\top}Ax$. Let $H=\E[aa^{\top}]$ be the Hessian of $f$. Let $R^2$ be the smallest positive real number such that:
\begin{equation}\label{eq:R_squared}
    \esp{\NRM{a}^2aa^{\top}} \preccurlyeq R^2 H \, .
\end{equation}
Further, similarly to \citet{jain2018accelerating}, we define the statistical condition number of the problem as the smallest $\Tilde{\kappa}>0$ such that:
\begin{equation}\label{eq:kappa_tilde}
    \esp{\NRM{a}^2_{H^{-1}}aa^{\top}} \preccurlyeq \Tilde{\kappa} H \, .
\end{equation}
\begin{theorem}[Continuized acceleration with pure multiplicative noise]\label{thm:multiplicative-noise} Assume that \eqref{eq:noiseless}, \eqref{eq:R_squared} and~\eqref{eq:kappa_tilde} hold true. Then the continuized acceleration satisfies the following.
\begin{enumerate}
\item  Choose the parameters $\eta_t = \frac{2}{t}, \eta_t' = 0, \gamma_t = \frac{1}{R^2}, \gamma_t' = \frac{t}{2R^2\Tilde{\kappa}}$. Then 
		\begin{align*}
			\frac{1}{2}\E \NRM{x_t-x_*}^2 \leq \frac{R^2\Tilde{\kappa}\Vert z_0 -x_* \Vert^2_{H^{-1}}}{t^2} \, .
		\end{align*}
\item Assume further that $f$ is $\mu$-strongly convex, i.e., all eigenvalues of $H$ are greater or equal to $\mu$, where $\mu > 0$. The condition number of $f$ is then defined as $\kappa=R^2/\mu$. For the parameters $\eta_t=\eta'_t=\frac{1}{\sqrt{\kappa\Tilde{\kappa}}}$, $\gamma_t=\frac{1}{R^2}$ and $\gamma'_t=\frac{1}{R^2}\sqrt{\frac{\kappa}{ \Tilde{\kappa}}}$, we have:
    \begin{equation*}
        \frac{1}{2}\E \NRM{x_t-x_*}^2\leq \left( \frac{1}{2}\NRM{x_0-x_*}^2+ \frac{\mu}{2}\NRM{z_0-x_*}^2_{H^{-1}}\right)\exp\left(-\frac{t}{\sqrt{\kappa \Tilde{\kappa}}}\right).
    \end{equation*}
\end{enumerate}
\end{theorem}
In the strongly convex case, the benefits of this acceleration are similar to those of \citet{jain2018accelerating} with classical discrete iterates: while stochastic gradient descent with stepsize $1/R^2$ is easily shown to achieve an exponential rate of convergence $1/\kappa$, the acceleration enjoys a rate of convergence of $1/\sqrt{\kappa\tilde{\kappa}}$. Note that from the definitions, $\Tilde{\kappa} \leq \kappa$, thus the acceleration performs as least as well as the naive algorithm. However, depending on the distribution of $a$, the improvement might either be significant or null. We refer the reader to the rich discussion in \citet{jain2018accelerating} which provides insights on the interpretation of $\tilde{\kappa}$ and on the possibility to accelerate. 
Below, we provide a complementary perspective on the statistical condition number in the context of gossip algorithms, where it can be interpreted in terms of effective resistances of graphs.

Albeit more restrictive in terms of assumptions, our analysis is much simpler than that of~\citet{jain2018accelerating}, as it relies on a standard Lyapunov function, similar to that of the continuized acceleration (Theorem \ref{thm:continuized}). In Appendix~\ref{app:coordinate_descent}, we use the same analysis framework to prove convergence of accelerated coordinate descent, which is another noiseless stochastic method.

\section{Accelerating Randomized Gossip\label{sec:gossip}}

The continuized framework allows designing accelerated decentralized algorithms requiring synchronized clocks, but no synchronization of the communications. In this section, we illustrate this statement in the simple case of gossip algorithms; the more general case of decentralized optimization is discussed in the next section.

Let $G=(\cV,\cE)$ a connected graph representing a communication network of agents. Each agent $v\in\cV$ is assigned a real number $x_0(v)\in\R$. The goal of the averaging (or gossip) problem is to design an iterative procedure allowing each agent of the network to know the average $\Bar{x}=\frac{1}{m}\sum_{v\in\cV}x_0(v)$ using only local communications, \emph{i.e.},~communications between adjacent agents in the network. 

We formalize the communication model of randomized gossip \cite{boyd2006randomized}. Time $t$ is indexed continuously in~$\R_{\ge0}$. We generate a Poisson point measure $\diff N(t,e)=\sum_{k\ge 1}\delta_{(T_k,\{v_k,w_k\})}$ with intensity measure $\diff t\otimes \cP$, where $\diff t$ is the Lebesgue measure on $\R_{\ge0}$ and $\cP = (\cP_{\{v,w\}})_{\{v,w\} \in \cE}$ is a probability measure on the set $\cE$ of edges. For $k\ge0$,  $T_k$ is a time at which edge $\{v_k,w_k\}$ is \emph{activated}: adjacent nodes $v_k$ and $w_k$ can communicate and perform a pairwise update. The Poisson point measure assumption implies that edges are activated independently of one another and from the past: the activation times of edge $\{v,w\}$ form a Poisson point process of intensity $\cP_{{\{v,w\}}}$. 

To solve the gossip problem, \citet{boyd2006randomized} proposed the following naive strategy: each agent $v\in\cV$ keeps a local estimate $x_t(v)$ of the average and, upon activation of edge $\{v_k,w_k\}$ at time $T_k\in\R_{\ge0}$, the activated nodes $v_k,w_k$ average their current estimates
\begin{equation*}
    x_{T_k}(v_k), \, x_{T_k}(w_k)\quad\longleftarrow\quad  \frac{x_{T_k-}(v_k)+x_{T_k-}(w_k)}{2} \, .
\end{equation*}
In this section, we accelerate this naive procedure. Our strategy is to apply Section~\ref{sec:stochastic} as follows. Consider the energy function
\begin{equation}
    \label{eq:f_gossip}
    f(x)=\sum_{\{v,w\}\in\cE} \frac{\cP_{\{v,w\}}}{2}(x(v)-x(w))^2 \, , \qquad x = (x(v))_{v\in\cV} \, . 
\end{equation} 
This function is convex, smooth, and writes in the form \eqref{eq:lst-sqr}:
\begin{equation} \label{eq:f_multiplicative_noise}
    f(x)=\E_{\{v,w\}\sim\cP}\left[\frac{1}{2} \left\langle x , a_{\{v,w\}}\right\rangle^2\right],
\end{equation}
where $a_{\{v,w\}}=e_v-e_w$ and $(e_v)_{v\in\cV}$ forms the canonical basis of $\R^{\cV}$. As in Section~\ref{sec:stochastic}, a stochastic gradient of $f$ is obtained by taking the gradient of one realization of the expectation, namely: 
\begin{equation} \label{eq:f_gossip_gradients}
\nabla f(x, {\{v,w\}})= \langle x , a_{\{v,w\}} \rangle a_{\{v,w\}} = 
    \begin{cases}
    x(v)-x(w) &\text{ at coordinate }v,\\
         x(w)-x(v) &\text{ at coordinate }w,\\
         0 &\text{ at all other coordinates}.
    \end{cases}
\end{equation}
As a consequence, a stochastic gradient step with stepsize $1/2$ corresponds to a local averaging alongside edge $\{v,w\}$, where $\{v,w\} \sim \cP$. More generally, the randomized gossip algorithm as described by \citet{boyd2006randomized} is the stochastic gradient descent:
\begin{equation} \label{eq:randomized_gossip_updates}
    \diff x_t = -\frac{1}{2}\int_{\R_{\ge0}\times \cE} \nabla f(x_t,\{v,w\})\diff N(t,\{v,w\}) \, .
\end{equation}
Using Section \ref{sec:stochastic}, we can accelerate this algorithm if we know the strong convexity parameter of $f$ and the constants $R^2$ and $\tilde{\kappa}$ as defined in \eqref{eq:R_squared} and \eqref{eq:kappa_tilde} respectively. These constants can be intepreted as graph-related quantities here.
\begin{definition}[Graph-related quantities]\label{def:graph_quantities}
The Laplacian matrix $\cL\in\R^{\cV\times \cV}$ of graph $G$ with weights $(\cP_{\{v,w\}})_{\{v,w\}\in\cE}$ on the edges is the matrix with entries $\cL_{v,w}=-\cP_{\{v,w\}}$ if $\{v,w\}\in\cE$, $\cL_{v,v}=\sum_{w\sim v}\cP_{\{v,w\}}$, and $\cL_{v,w}=0$ if $\{v,w\}\notin \cE$. We denote $\mu_{\rm gossip}$ the second smallest eigenvalue of $\cL$, corresponding to its smallest positive eigenvalue. For $\{v,w\}\in\cE$, let $R_{\rm eff}(v,w)=(e(v)-e(w))^{\top} \cL^{-1} (e(v)-e(w))$ be the effective resistance of edge $\{v,w\}$, and $R_{\max}=\max_{\{v,w\}\in\cE} R_{\rm eff}(v,w)$ be the maximal resistance in the graph. 
\end{definition}
The function $f$ is quadratic with Hessian $\cL$, and strongly convex with parameter $\mu_{\rm gossip}$ on the hyperplane $F=\{x\in\R^{\cV}:\sum_{v\in\cV}x(v)=\bar{x}\}$; hence we use the (perhaps abusive) notation $\mu_{\rm gossip}$ throughout. Moreover, the conditions \eqref{eq:R_squared} and \eqref{eq:kappa_tilde} are satisfied with $R^2=2$, $\Tilde{\kappa}=R_{\max}$. 

These parameters being given, the accelerated stochastic gradient descent updates~\eqref{eq:continuized-sgd-additive-1}-\eqref{eq:continuized-sgd-additive-2} can be instantiated as follows. Each agent $v\in\cV$ keeps two local estimates $x_t(v),z_t(v)$ of $\Bar{x}$, initialized at $x_0(v)$. Upon activation of edge $\{v_k,w_k\}$ at time $T_k$,
\begin{align*}
    x_{T_k}(v_k)&=x_{T_k}(w_k)=\frac{x_{T_k-}(v_k)+x_{T_k-}(w_k)}{2} \, ,\\
    z_{T_k}(v_k)&=z_{T_k-}(v_k)+\frac{1}{\sqrt{2\mu_{\rm gossip}R_{\max}}}(x_{T_k-}(w_k)-x_{T_k-}(v_k)) \, ,\\
    z_{T_k}(w_k)&=z_{T_k-}(w_k)+\frac{1}{\sqrt{2\mu_{\rm gossip}R_{\max}}}(x_{T_k-}(v_k)-x_{T_k-}(w_k)) \, .
\end{align*}
Between these updates, $x_t(v)$ and $z_t(v)$ locally mix at all nodes $v\in\cV$, according to the coupled ODE:
\begin{align*}
    \diff x_t(v)&= \sqrt{\frac{2\mu_{\rm gossip}}{R_{\max}}}(z_t(v)-x_t(v))\diff t,\\
    \diff z_t(v)&= \sqrt{\frac{2\mu_{\rm gossip}}{R_{\max}}}(x_t(v)-z_t(v))\diff t.
\end{align*}
This algorithm is \emph{asynchronous} in the sense that it does not require global synchronous operations: the mixing of local variables does not require any synchronization since parameter $t\in\R_{\geq0}$ is available at all nodes independently from the number of past updates, while a local pairwise update between adjacent nodes $v$ and $w$ only requires a local synchronization.
\begin{theorem}[Accelerated randomized gossip]\label{thm:gossip} Let $(x_t(v))_{v\in\cV,t\ge0}$ be generated with accelerated randomized gossip. For any $t\in\R_{\ge0}$:
\begin{equation*}
    \sum_{v\in\cV}\frac{1}{2}\esp{\big(x_t(v)-\Bar{x}\big)^2} \leq 2\left(\sum_{v\in\cV}\frac{1}{2}\big(x_0(v)-\Bar{x}\big)^2\right)\exp\left(-\sqrt{\frac{\mu_{\rm gossip}}{2R_{\max}}}t\right).
\end{equation*}
\end{theorem}
Let $\theta_{\rm ARG}=\sqrt{\frac{\mu_{\rm gossip}}{2R_{\max}}}$ be the rate of convergence of accelerated randomized gossip, and $\theta_{\rm RG}=\mu_{\rm gossip}$ be the rate of convergence of randomized gossip \citep{boyd2006randomized}. We have $\theta_{\rm ARG}\geq \theta_{\rm RG}/\sqrt{2}$. Let us exhibit scenarios over which accelerated randomized gossip gains several orders of magnitude. Denoting $\cP_{\min}=\min_{\{v,w\}\in\cE} \cP_{{\{v,w\}}}$, \citet{ELLENS20112491} ensures that for $\{v,w\}\in\cE$, $\cP_{\min}R_{\rm eff}(v,w)\le 1$, so that $R_{\max}\leq \cP_{\min}^{-1}$.
\begin{cor}[Comparison with randomized gossip] \label{cor:accelerated_gossip} Accelerated randomized gossip achieves a rate satisfying:
\begin{equation*}
    \sqrt{\frac{\theta_{\rm RG}\cP_{\min}}{2}}\leq \theta_{\rm ARG}.
\end{equation*}
Assume furthermore that there exist some constants $c>0$ such that for all $\{v,w\}\in\cE$, $\cP_{\{v,w\}}\le c \cP_{\min}$ and $d_v+d_w \le 2d$. Then, with $C=1/\sqrt{2cd}$:
\begin{equation*}
    C\sqrt{\frac{\theta_{\rm RG}}{|\cV|}} \leq \theta_{\rm ARG}.
\end{equation*}
\end{cor}
Assume now for simplicity that the Poisson intensities $\cP_{\edgeuv}$ are all equal to $1/|\cE|$.
Denoting $|\cV|=m$, on the cyclic and the line graph, this gives us $\theta_{\rm ARG} = \Omega(1/m^2)$ while $\theta_{\rm RG}\asymp 1/m^3$. On a $d$-dimensional grid, we have $\theta_{\rm ARG}= \Omega(1/m^{1+1/d})$ and $\theta_{\rm RG}\asymp 1/m^{1+2/d}$. However, on graphs with unbounded degrees, no improvements are observed, as illustrated in Figure~\ref{fig:gossip}, Appendix~\ref{app:simulation_gossip}. In the case of the complete graph, this is expected since at least $\theta_{\rm RG}^{-1}\asymp m$ communications are needed to compute the average.
We thus recover the same rates as \citet{Dimakis_2008} for the graphs they study, but generalized to any network.


\section{Accelerating Asynchronous Decentralized Optimization}
\label{sec:decentralized_opt}
Our continuized framework for accelerating randomized gossip can be extended to the more general problem of decentralized optimization: each node $v$ in the network $G$ previously defined holds a function $f_v:\R^d\to\R$, $\mu$-strongly convex and $L$-smooth. Nodes of the network collaborate to solve:
\begin{equation}\label{eq:pbm_optim_decentralisé}
    \min_{x\in\R^d} \left\{ f(x)=\frac{1}{|\cV|}\sum_{v\in\cV} f_v(x)\right\}.
\end{equation}
As in gossip averaging, only local communications are allowed. Quantities related to $f_v$ can only be computed at node $v$. In the case of empirical risk minimization, $f_v$ represents the empirical risk related to node $v$'s local data. Setting $f_v(x)=\frac{1}{2}\NRM{x-x_0(v)}^2$ leads to the averaging problem previously described. 
Similarly to Section \ref{sec:gossip}, time is indexed continuously by $t$ in $\R_{\geq0}$, and communications are ruled by the same Poisson point measure $\diff N(t,e)=\sum_{k\geq1}\delta_{(T_k,\{v_k,w_k\})}$ on $\R_{\geq0}\times \cE$. 
Yet, we no longer assume (as in Theorem~\ref{thm:multiplicative-noise}) that the function $f$ is quadratic. Instead, we write a dual formulation of Problem~\eqref{eq:pbm_optim_decentralisé} and minimize it using a continuized version of accelerated coordinate descent~\citep{neststich2017acdm} that we present in Appendix~\ref{app:coordinate_descent}. This leads to an accelerated decentralized algorithm to solve \eqref{eq:pbm_optim_decentralisé}. Our algorithm mimics the behavior of accelerated randomized gossip: a node possesses two local parameters that mix continuously through a time-independent ODE. At time $T_k$, adjacent nodes $v_k$ and $w_k$ use their local function in order to compute gradient conjugates $\nabla f^*_v(x(v)),\nabla f^*_w(x(w))$. Since the local functions are not simple quadratics anymore, the stochastic gradients $\nabla f(x, \edgeuv)$ from Equation~\eqref{eq:f_gossip_gradients} are replaced by terms proportional to:
\begin{equation} \label{eq:f_gossip_gradients}
G(y, \edgeuv) = 
    \begin{cases}
    \nabla f^*_v(y(v))-\nabla f^*_w(y(w)) &\text{ at coordinate }v,\\
         -\nabla f^*_v(y(v))-\nabla f^*_w(y(w)) &\text{ at coordinate }w,\\
         0 &\text{ at all other coordinates}.
    \end{cases}
\end{equation}
Due to lack of space, we describe the iterations more in details in Appendix~\ref{app:decentralized_optim}, together with a relevant choice of parameters. The crucial point is that, similarly to the gossip averaging case, we do not require nodes to be aware of a global iteration counter. Yet, we still obtain the same convergence rate as \citep{hendrikx2018accelerated}, as provided by the following theorem. The same approach can be used to ``continuize'' other accelerated randomized gossip algorithms for decentralized optimization, such as ADFS~\citep{hendrikx2019accelerated}. 
\begin{theorem}[Accelerated asynchronous decentralized optimization] \label{thm:acc_asy_decentralized_optim}
For $(x_t(v))_{v\in\cV} = (\nabla f_v^*(z_t(v)))_{v\in\cV}$ generated by the accelerated coordinate descent on the dual of Problem~\eqref{eq:pbm_optim_decentralisé}: 
\begin{equation*}
     \sum_{v\in\cV}\frac{1}{2}\esp{\NRM{x_t(v)-x_*}^2} \leq C\left(\sum_{v\in\cV}\frac{1}{2}\NRM{x_0(v)-x_*}^2\right)\exp\left(- \frac{\theta'_{\rm ARG}}{\sqrt{\kappa}}t\right)\,,
\end{equation*}
where $\kappa=\mu/L$ is an upper bound on the condition number of $f$, $C$ is a constant that depends on the graph and $\kappa$, and $\theta'_{\rm ARG}$ is the rate of convergence of accelerated randomized gossip on the graph $G$ as defined in Theorem \ref{thm:gossip} but with graph resistances are defined in a different way (see Theorem \ref{thm:decentralized_optim_appendix}). 
\end{theorem}

\section{Conclusion}
In this work, we introduced a continuized version of Nesterov's accelerated gradients. In a nutshell, the method has two sequences of iterates which take gradient steps at random times. In between gradient steps, the two sequences mix following a simple ordinary differential equation, whose parameters are picked to ensure good convergence properties of the method.

As compared to other continuous time models of Nesterov acceleration, a key feature of this approach is that the method can be implemented without any approximation, as the differential equation governing the mixing procedure has a simple analytical solution. A discretization of the continuized method corresponds to an accelerated gradient method with random parameters.  

Continuization strategies were introduced in the context of Markov chains~\citep{aldous1995reversible}. Here, they allow using acceleration mechanisms in asynchronous distributed optimization, where usually agents are not aware of the total number of iterations taken so far. This is showcased in the context of asynchronous gossip algorithms.

\textbf{Acknowledgements:} The authors thank Sam Power for pointing out the class of piecewise deterministic Markov processes and related references, and an anonymous reviewer for suggesting Remark \ref{rmk:geometric}. This work was funded in part by the French government under management of Agence Nationale de la Recherche as part of the “Investissements d’avenir” program, reference ANR-19-P3IA-0001(PRAIRIE 3IA Institute). We also acknowledge support from the European Research Council (grant SEQUOIA 724063), from the DGA, and from the MSR-INRIA joint centre.

\bibliographystyle{plainnat} 
\bibliography{bibliography}

\begin{thebibliography}{57}
\providecommand{\natexlab}[1]{#1}
\providecommand{\url}[1]{\texttt{#1}}
\expandafter\ifx\csname urlstyle\endcsname\relax
  \providecommand{\doi}[1]{doi: #1}\else
  \providecommand{\doi}{doi: \begingroup \urlstyle{rm}\Url}\fi

\bibitem[Aldous and Fill(2002)]{aldous1995reversible}
David Aldous and James~Allen Fill.
\newblock Reversible markov chains and random walks on graphs.
\newblock 2002.
\newblock Unfinished monograph, recompiled 2014, available at
  \url{http://www.stat.berkeley.edu/$\sim$aldous/RWG/book.html}.

\bibitem[{Allen-Zhu} and Orecchia(2017)]{allen2014linear}
Zeyuan {Allen-Zhu} and Lorenzo Orecchia.
\newblock {Linear Coupling: An Ultimate Unification of Gradient and Mirror
  Descent}.
\newblock In \emph{Proceedings of the 8th Innovations in Theoretical Computer
  Science}, ITCS~'17, 2017.

\bibitem[Arjevani et~al.(2016)Arjevani, Shalev-Shwartz, and
  Shamir]{arjevani2015lower}
Yossi Arjevani, Shai Shalev-Shwartz, and Ohad Shamir.
\newblock On lower and upper bounds in smooth and strongly convex optimization.
\newblock \emph{Journal of Machine Learning Research}, 17\penalty0
  (126):\penalty0 1--51, 2016.

\bibitem[Attouch et~al.(2018)Attouch, Chbani, Peypouquet, and
  Redont]{attouch2018fast}
Hedy Attouch, Zaki Chbani, Juan Peypouquet, and Patrick Redont.
\newblock Fast convergence of inertial dynamics and algorithms with asymptotic
  vanishing viscosity.
\newblock \emph{Mathematical Programming}, 168\penalty0 (1):\penalty0 123--175,
  2018.

\bibitem[Attouch et~al.(2019)Attouch, Chbani, and Riahi]{attouch2019rate}
Hedy Attouch, Zaki Chbani, and Hassan Riahi.
\newblock Rate of convergence of the \uppercase{N}esterov accelerated gradient
  method in the subcritical case $\alpha \leq 3$.
\newblock \emph{ESAIM: Control, Optimisation and Calculus of Variations},
  25:\penalty0 2, 2019.

\bibitem[Aybat et~al.(2020)Aybat, Fallah, Gurbuzbalaban, and
  Ozdaglar]{aybat2020robust}
Necdet~Serhat Aybat, Alireza Fallah, Mert Gurbuzbalaban, and Asuman Ozdaglar.
\newblock Robust accelerated gradient methods for smooth strongly convex
  functions.
\newblock \emph{SIAM Journal on Optimization}, 30\penalty0 (1):\penalty0
  717--751, 2020.

\bibitem[Berthier et~al.(2020{\natexlab{a}})Berthier, Bach, and
  Gaillard]{berthier2020accelerated}
Rapha{\"e}l Berthier, Francis Bach, and Pierre Gaillard.
\newblock Accelerated gossip in networks of given dimension using jacobi
  polynomial iterations.
\newblock \emph{SIAM Journal on Mathematics of Data Science}, 2\penalty0
  (1):\penalty0 24--47, 2020{\natexlab{a}}.

\bibitem[Berthier et~al.(2020{\natexlab{b}})Berthier, Bach, and
  Gaillard]{berthier2020tight}
Rapha\"{e}l Berthier, Francis Bach, and Pierre Gaillard.
\newblock Tight nonparametric convergence rates for stochastic gradient descent
  under the noiseless linear model.
\newblock In \emph{Advances in Neural Information Processing Systems},
  volume~33, pages 2576--2586, 2020{\natexlab{b}}.

\bibitem[Betancourt et~al.(2018)Betancourt, Jordan, and
  Wilson]{betancourt2018symplectic}
Michael Betancourt, Michael Jordan, and Ashia Wilson.
\newblock On symplectic optimization.
\newblock \emph{arXiv preprint arXiv:1802.03653}, 2018.

\bibitem[Bottou et~al.(2018)Bottou, Curtis, and
  Nocedal]{bottou2018optimization}
L{\'e}on Bottou, Frank~E Curtis, and Jorge Nocedal.
\newblock Optimization methods for large-scale machine learning.
\newblock \emph{SIAM Review}, 60\penalty0 (2):\penalty0 223--311, 2018.

\bibitem[Boyd et~al.(2006)Boyd, Ghosh, Prabhakar, and Shah]{boyd2006randomized}
Stephen Boyd, Arpita Ghosh, Balaji Prabhakar, and Devavrat Shah.
\newblock Randomized gossip algorithms.
\newblock \emph{IEEE Transactions on Information Theory}, 52\penalty0
  (6):\penalty0 2508--2530, 2006.

\bibitem[Bubeck et~al.(2015)Bubeck, Lee, and Singh]{bubeck2015geometric}
S{\'e}bastien Bubeck, Yin~Tat Lee, and Mohit Singh.
\newblock A geometric alternative to {N}esterov's accelerated gradient descent.
\newblock \emph{arXiv preprint arXiv:1506.08187}, 2015.

\bibitem[Cohen et~al.(2018)Cohen, Diakonikolas, and
  Orecchia]{cohen2018acceleration}
Michael Cohen, Jelena Diakonikolas, and Lorenzo Orecchia.
\newblock On acceleration with noise-corrupted gradients.
\newblock In \emph{Proceedings of the 35th International Conference on Machine
  Learning}, volume~80 of \emph{Proceedings of Machine Learning Research},
  pages 1019--1028. PMLR, 2018.

\bibitem[d'Aspremont et~al.(2021)d'Aspremont, Scieur, and
  Taylor]{daspremont2021acceleration}
Alexandre d'Aspremont, Damien Scieur, and Adrien Taylor.
\newblock Acceleration methods.
\newblock 2021.

\bibitem[Davis(1984)]{davis1984piecewise}
Mark~HA Davis.
\newblock Piecewise-deterministic markov processes: a general class of
  non-diffusion stochastic models.
\newblock \emph{Journal of the Royal Statistical Society: Series B
  (Methodological)}, 46\penalty0 (3):\penalty0 353--376, 1984.

\bibitem[Davis(2018)]{davis2018markov}
Mark~HA Davis.
\newblock \emph{Markov models \& optimization}.
\newblock Routledge, 2018.

\bibitem[Devolder(2011)]{devolder2011stochastic}
Olivier Devolder.
\newblock Stochastic first order methods in smooth convex optimization.
\newblock Technical report, CORE, 2011.

\bibitem[Diakonikolas and Orecchia(2019)]{diakonikolas2019approximate}
Jelena Diakonikolas and Lorenzo Orecchia.
\newblock The approximate duality gap technique: A unified theory of
  first-order methods.
\newblock \emph{SIAM Journal on Optimization}, 29\penalty0 (1):\penalty0
  660--689, 2019.

\bibitem[{Dimakis} et~al.(2010){Dimakis}, {Kar}, {Moura}, {Rabbat}, and
  {Scaglione}]{dimakis2010synchgossip}
A.~G. {Dimakis}, S.~{Kar}, J.~M.~F. {Moura}, M.~G. {Rabbat}, and
  A.~{Scaglione}.
\newblock Gossip algorithms for distributed signal processing.
\newblock \emph{Proceedings of the IEEE}, 98\penalty0 (11):\penalty0
  1847--1864, 2010.

\bibitem[Dimakis et~al.(2008{\natexlab{a}})Dimakis, Sarwate, and
  Wainwright]{Dimakis_2008}
Alexandros D.~G. Dimakis, Anand~D. Sarwate, and Martin~J. Wainwright.
\newblock Geographic gossip: Efficient averaging for sensor networks.
\newblock \emph{IEEE Transactions on Signal Processing}, 56\penalty0
  (3):\penalty0 1205–1216, 2008{\natexlab{a}}.
\newblock ISSN 1053-587X.
\newblock \doi{10.1109/tsp.2007.908946}.
\newblock URL \url{http://dx.doi.org/10.1109/TSP.2007.908946}.

\bibitem[Dimakis et~al.(2008{\natexlab{b}})Dimakis, Sarwate, and
  Wainwright]{dimakis2008geographic}
Alexandros~DG Dimakis, Anand~D Sarwate, and Martin~J Wainwright.
\newblock Geographic gossip: Efficient averaging for sensor networks.
\newblock \emph{IEEE Transactions on Signal Processing}, 56\penalty0
  (3):\penalty0 1205--1216, 2008{\natexlab{b}}.

\bibitem[Ellens et~al.(2011)Ellens, Spieksma, {Van Mieghem}, Jamakovic, and
  Kooij]{ELLENS20112491}
W.~Ellens, F.M. Spieksma, P.~{Van Mieghem}, A.~Jamakovic, and R.E. Kooij.
\newblock Effective graph resistance.
\newblock \emph{Linear Algebra and its Applications}, 435\penalty0
  (10):\penalty0 2491--2506, 2011.
\newblock ISSN 0024-3795.
\newblock \doi{https://doi.org/10.1016/j.laa.2011.02.024}.
\newblock URL
  \url{https://www.sciencedirect.com/science/article/pii/S0024379511001443}.
\newblock Special Issue in Honor of Dragos Cvetkovic.

\bibitem[Even et~al.(2021)Even, Hendrikx, and
  Massouli{\'e}]{even2021decentralized}
Mathieu Even, Hadrien Hendrikx, and Laurent Massouli{\'e}.
\newblock Decentralized optimization with heterogeneous delays: a
  continuous-time approach.
\newblock \emph{arXiv preprint arXiv:2106.03585}, 2021.

\bibitem[Flammarion and Bach(2015)]{flammarion2015averaging}
Nicolas Flammarion and Francis Bach.
\newblock From averaging to acceleration, there is only a step-size.
\newblock In \emph{Conference on Learning Theory}, pages 658--695. PMLR, 2015.

\bibitem[Hendrikx et~al.(2018)Hendrikx, Bach, and
  Massoulié]{hendrikx2018accelerated}
Hadrien Hendrikx, Francis Bach, and Laurent Massoulié.
\newblock Accelerated decentralized optimization with local updates for smooth
  and strongly convex objectives, 2018.

\bibitem[Hendrikx et~al.(2019)Hendrikx, Bach, and
  Massouli{\'e}]{hendrikx2019accelerated}
Hadrien Hendrikx, Francis Bach, and Laurent Massouli{\'e}.
\newblock An accelerated decentralized stochastic proximal algorithm for finite
  sums.
\newblock \emph{arXiv preprint arXiv:1905.11394}, 2019.

\bibitem[Hu et~al.(2009)Hu, Pan, and Kwok]{hu2009accelerated}
Chonghai Hu, Weike Pan, and James Kwok.
\newblock Accelerated gradient methods for stochastic optimization and online
  learning.
\newblock In \emph{Advances in Neural Information Processing Systems},
  volume~22, pages 781--789, 2009.

\bibitem[Ikeda and Watanabe(2014)]{ikeda2014stochastic}
Nobuyuki Ikeda and Shinzo Watanabe.
\newblock \emph{Stochastic differential equations and diffusion processes}.
\newblock Elsevier, 2014.

\bibitem[Jacod and Shiryaev(2013)]{jacod2013limit}
Jean Jacod and Albert Shiryaev.
\newblock \emph{Limit theorems for stochastic processes}, volume 288.
\newblock Springer Science \& Business Media, 2013.

\bibitem[Jain et~al.(2018)Jain, Kakade, Kidambi, Netrapalli, and
  Sidford]{jain2018accelerating}
Prateek Jain, Sham~M Kakade, Rahul Kidambi, Praneeth Netrapalli, and Aaron
  Sidford.
\newblock Accelerating stochastic gradient descent for least squares
  regression.
\newblock In \emph{Conference On Learning Theory}, pages 545--604, 2018.

\bibitem[Kakade et~al.(2009)Kakade, Shalev-Shwartz, and
  Tewari]{kakade2009duality}
Sham Kakade, Shai Shalev-Shwartz, and Ambuj Tewari.
\newblock On the duality of strong convexity and strong smoothness: Learning
  applications and matrix regularization.
\newblock \emph{Unpublished Manuscript, \url{http://ttic.
  uchicago.edu/shai/papers/KakadeShalevTewari09.pdf}}, 2009.

\bibitem[Kim and Fessler(2016)]{kim2016optimized}
Donghwan Kim and Jeffrey~A Fessler.
\newblock Optimized first-order methods for smooth convex minimization.
\newblock \emph{Mathematical programming}, 159\penalty0 (1):\penalty0 81--107,
  2016.

\bibitem[Krichene et~al.(2015)Krichene, Bayen, and
  Bartlett]{krichene2015accelerated}
Walid Krichene, Alexandre Bayen, and Peter Bartlett.
\newblock Accelerated mirror descent in continuous and discrete time.
\newblock \emph{Advances in Neural Information Processing Systems},
  28:\penalty0 2845--2853, 2015.

\bibitem[Lan(2012)]{lan2012optimal}
Guanghui Lan.
\newblock An optimal method for stochastic composite optimization.
\newblock \emph{Math. Program.}, 133\penalty0 (1-2, Ser. A):\penalty0 365--397,
  2012.

\bibitem[Latz(2021)]{latz2021analysis}
Jonas Latz.
\newblock Analysis of stochastic gradient descent in continuous time.
\newblock \emph{Statistics and Computing}, 31\penalty0 (4):\penalty0 1--25,
  2021.

\bibitem[Le~Gall(2016)]{le2016brownian}
Jean-Fran{\c{c}}ois Le~Gall.
\newblock \emph{Brownian Motion, Martingales, and Stochastic Calculus}, volume
  274.
\newblock Springer, 2016.

\bibitem[Liu et~al.(2013)Liu, Anderson, Cao, and Morse]{LIU2013873}
Ji~Liu, Brian~D.O. Anderson, Ming Cao, and A.~Stephen Morse.
\newblock Analysis of accelerated gossip algorithms.
\newblock \emph{Automatica}, 49\penalty0 (4):\penalty0 873 -- 883, 2013.
\newblock ISSN 0005-1098.
\newblock \doi{https://doi.org/10.1016/j.automatica.2013.01.001}.
\newblock URL
  \url{http://www.sciencedirect.com/science/article/pii/S0005109813000022}.

\bibitem[Loizou et~al.(2019)Loizou, Rabbat, and
  Richt{\'a}rik]{loizou2019provably}
Nicolas Loizou, Michael Rabbat, and Peter Richt{\'a}rik.
\newblock Provably accelerated randomized gossip algorithms.
\newblock In \emph{ICASSP 2019-2019 ieee international conference on acoustics,
  speech and signal processing (icassp)}, pages 7505--7509. IEEE, 2019.

\bibitem[Montijano et~al.(2011)Montijano, Montijano, and
  Sagues]{montijano2011chebgossip}
Eduardo Montijano, Juan Montijano, and C.~Sagues.
\newblock Chebyshev polynomials in distributed consensus applications.
\newblock \emph{IEEE Transactions on Signal Processing}, 61, 11 2011.
\newblock \doi{10.1109/TSP.2012.2226173}.

\bibitem[Muehlebach and Jordan(2019)]{muehlebach2019dynamical}
Michael Muehlebach and Michael Jordan.
\newblock A dynamical systems perspective on \uppercase{N}esterov acceleration.
\newblock In \emph{International Conference on Machine Learning}, pages
  4656--4662. PMLR, 2019.

\bibitem[Nemirovskij and Yudin(1983)]{nemirovskij1983problem}
Arkadij~Semenovi{\v{c}} Nemirovskij and David~Borisovich Yudin.
\newblock \emph{Problem Complexity and Method Efficiency in Optimization}.
\newblock Wiley-Interscience, 1983.

\bibitem[Nesterov(1983)]{nesterov1983method}
Yurii Nesterov.
\newblock {A method of solving a convex programming problem with convergence
  rate $O (1/k^2)$}.
\newblock \emph{Dokl. Akad. Nauk SSSR}, 27\penalty0 (2):\penalty0 372--376,
  1983.

\bibitem[Nesterov(2003)]{nesterov2003introductory}
Yurii Nesterov.
\newblock \emph{Introductory Lectures on Convex Optimization: A Basic Course},
  volume~87.
\newblock Springer Science \& Business Media, 2003.

\bibitem[Nesterov(2012)]{nesterov2012efficiency}
Yurii Nesterov.
\newblock Efficiency of coordinate descent methods on huge-scale optimization
  problems.
\newblock \emph{SIAM Journal on Optimization}, 22\penalty0 (2):\penalty0
  341--362, 2012.

\bibitem[Nesterov and Stich(2017)]{neststich2017acdm}
Yurii Nesterov and Sebastian~U. Stich.
\newblock Efficiency of the accelerated coordinate descent method on structured
  optimization problems.
\newblock \emph{SIAM Journal on Optimization}, 27\penalty0 (1):\penalty0
  110--123, 2017.
\newblock \doi{10.1137/16M1060182}.
\newblock URL \url{https://doi.org/10.1137/16M1060182}.

\bibitem[Sanz-Serna and Zygalakis(2020)]{sanz2020connections}
Jes\'us~Mar\'ia Sanz-Serna and Konstantinos Zygalakis.
\newblock The connections between \uppercase{L}yapunov functions for some
  optimization algorithms and differential equations.
\newblock \emph{arXiv preprint arXiv:2009.00673}, 2020.

\bibitem[Scaman et~al.(2017)Scaman, Bach, Bubeck, Lee, and
  Massouli{\'e}]{scaman2017optimal}
Kevin Scaman, Francis Bach, S{\'e}bastien Bubeck, Yin~Tat Lee, and Laurent
  Massouli{\'e}.
\newblock Optimal algorithms for smooth and strongly convex distributed
  optimization in networks.
\newblock In \emph{International Conference on Machine Learning}, 2017.

\bibitem[Shi et~al.(2018)Shi, Du, Jordan, and Su]{shi2018understanding}
Bin Shi, Simon Du, Michael Jordan, and Weijie Su.
\newblock Understanding the acceleration phenomenon via high-resolution
  differential equations.
\newblock \emph{arXiv preprint arXiv:1810.08907}, 2018.

\bibitem[Shi et~al.(2019)Shi, Du, Su, and Jordan]{shi2019acceleration}
Bin Shi, Simon Du, Weijie Su, and Michael Jordan.
\newblock Acceleration via symplectic discretization of high-resolution
  differential equations.
\newblock In \emph{Advances in Neural Information Processing Systems},
  volume~32, pages 5744--5752, 2019.

\bibitem[Su et~al.(2014)Su, Boyd, and Candes]{su2014differential}
Weijie Su, Stephen Boyd, and Emmanuel Candes.
\newblock A differential equation for modeling \uppercase{N}esterov’s
  accelerated gradient method: theory and insights.
\newblock \emph{Advances in neural information processing systems},
  27:\penalty0 2510--2518, 2014.

\bibitem[Tseng and Yun(2009)]{tseng2009coordinate}
Paul Tseng and Sangwoon Yun.
\newblock A coordinate gradient descent method for nonsmooth separable
  minimization.
\newblock \emph{Mathematical Programming}, 117\penalty0 (1):\penalty0 387--423,
  2009.

\bibitem[Vaswani et~al.(2019)Vaswani, Bach, and Schmidt]{vaswani2019fast}
Sharan Vaswani, Francis Bach, and Mark Schmidt.
\newblock Fast and faster convergence of sgd for over-parameterized models and
  an accelerated perceptron.
\newblock In \emph{The 22nd International Conference on Artificial Intelligence
  and Statistics}, pages 1195--1204. PMLR, 2019.

\bibitem[Wibisono et~al.(2016)Wibisono, Wilson, and
  Jordan]{wibisono2016variational}
Andre Wibisono, Ashia~C Wilson, and Michael~I Jordan.
\newblock A variational perspective on accelerated methods in optimization.
\newblock \emph{Proceedings of the National Academy of Sciences}, 113\penalty0
  (47):\penalty0 E7351--E7358, 2016.

\bibitem[Wilson et~al.(2016)Wilson, Recht, and Jordan]{wilson2016lyapunov}
Ashia Wilson, Benjamin Recht, and Michael~I Jordan.
\newblock A \uppercase{L}yapunov analysis of momentum methods in optimization.
\newblock \emph{arXiv preprint arXiv:1611.02635}, 2016.

\bibitem[Wright(2015)]{wright2015coordinate}
Stephen Wright.
\newblock Coordinate descent algorithms.
\newblock \emph{Math. Program.}, 151\penalty0 (1, Ser. B):\penalty0 3--34,
  2015.

\bibitem[Xiao(2010)]{xia2010dual}
Lin Xiao.
\newblock {Dual averaging methods for regularized stochastic learning and
  online optimization}.
\newblock \emph{J. Mach. Learn. Res.}, 11:\penalty0 2543--2596, 2010.

\bibitem[Zhang et~al.(2018)Zhang, Mokhtari, Sra, and
  Jadbabaie]{zhang2018direct}
Jingzhao Zhang, Aryan Mokhtari, Suvrit Sra, and Ali Jadbabaie.
\newblock Direct \uppercase{R}unge-\uppercase{K}utta discretization achieves
  acceleration.
\newblock In \emph{Advances in Neural Information Processing Systems},
  volume~31, pages 3900--3909, 2018.

\end{thebibliography}

\appendix

\newpage

\section{Numerical Simulations}

\subsection{Simulations of the discretized continuized acceleration}
\label{ap:simus-cont}

\begin{figure}[h]
	\includegraphics[width=0.49\linewidth]{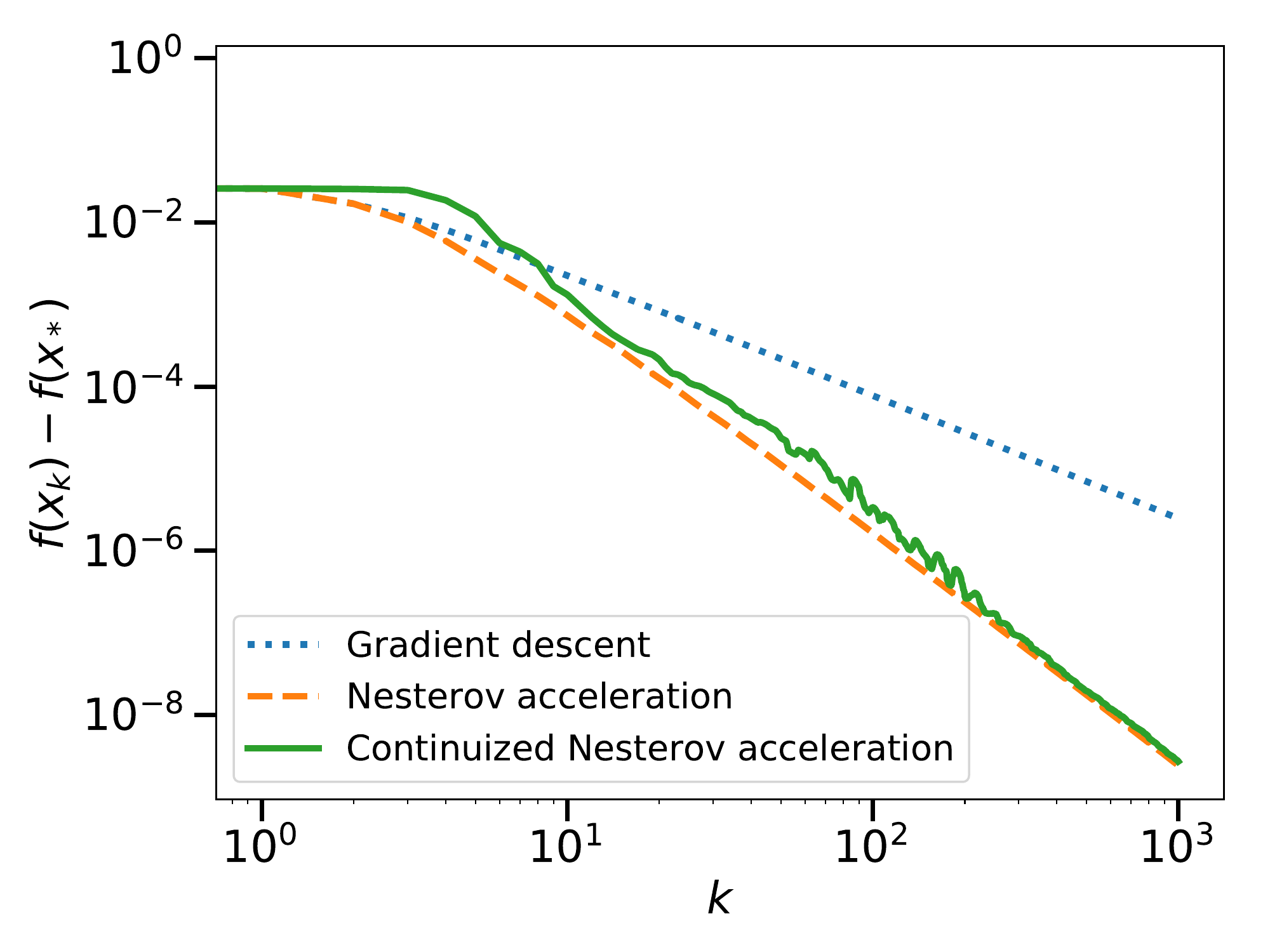}
	\includegraphics[width=0.49\linewidth]{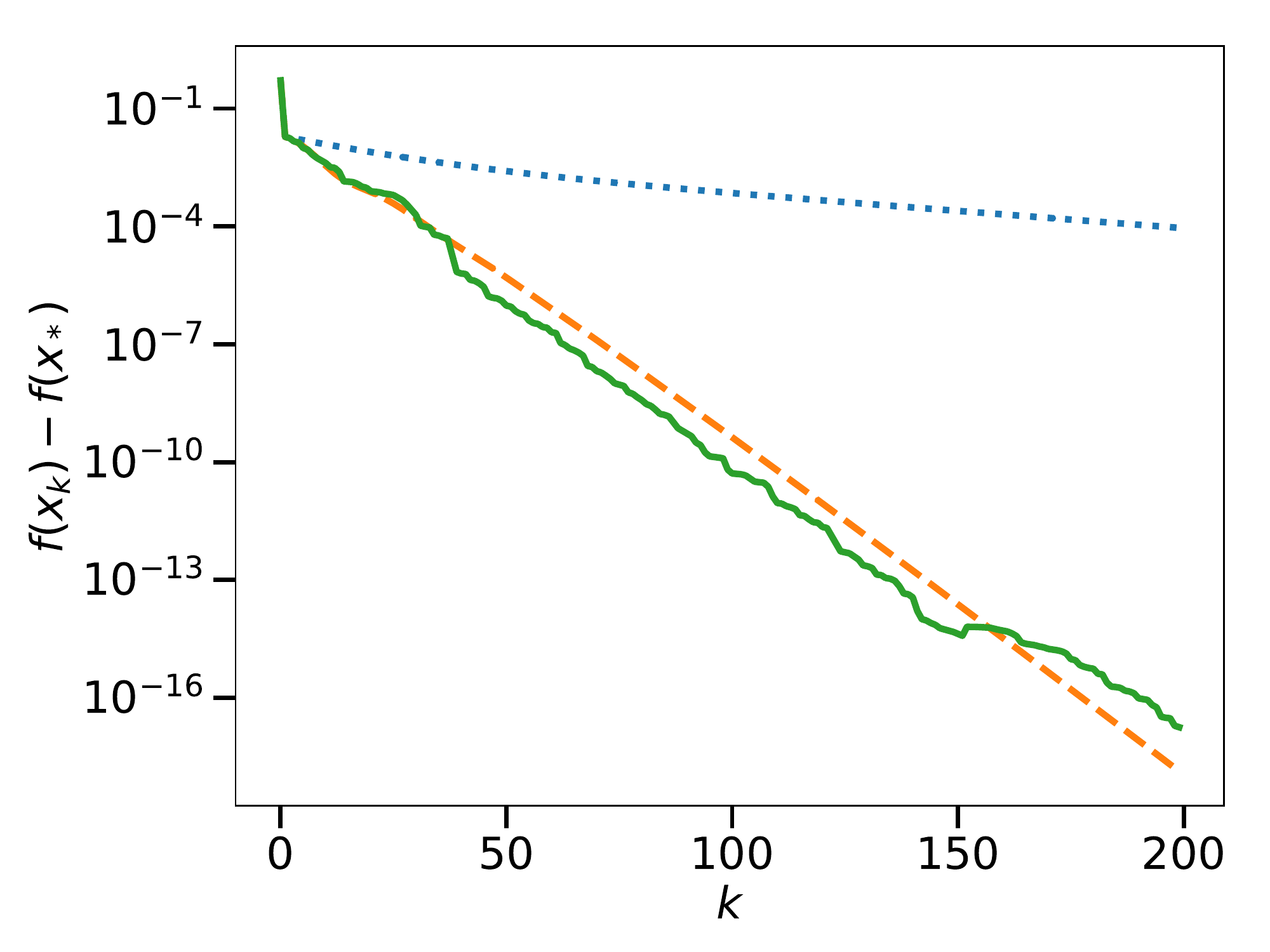}
	\includegraphics[width=0.49\linewidth]{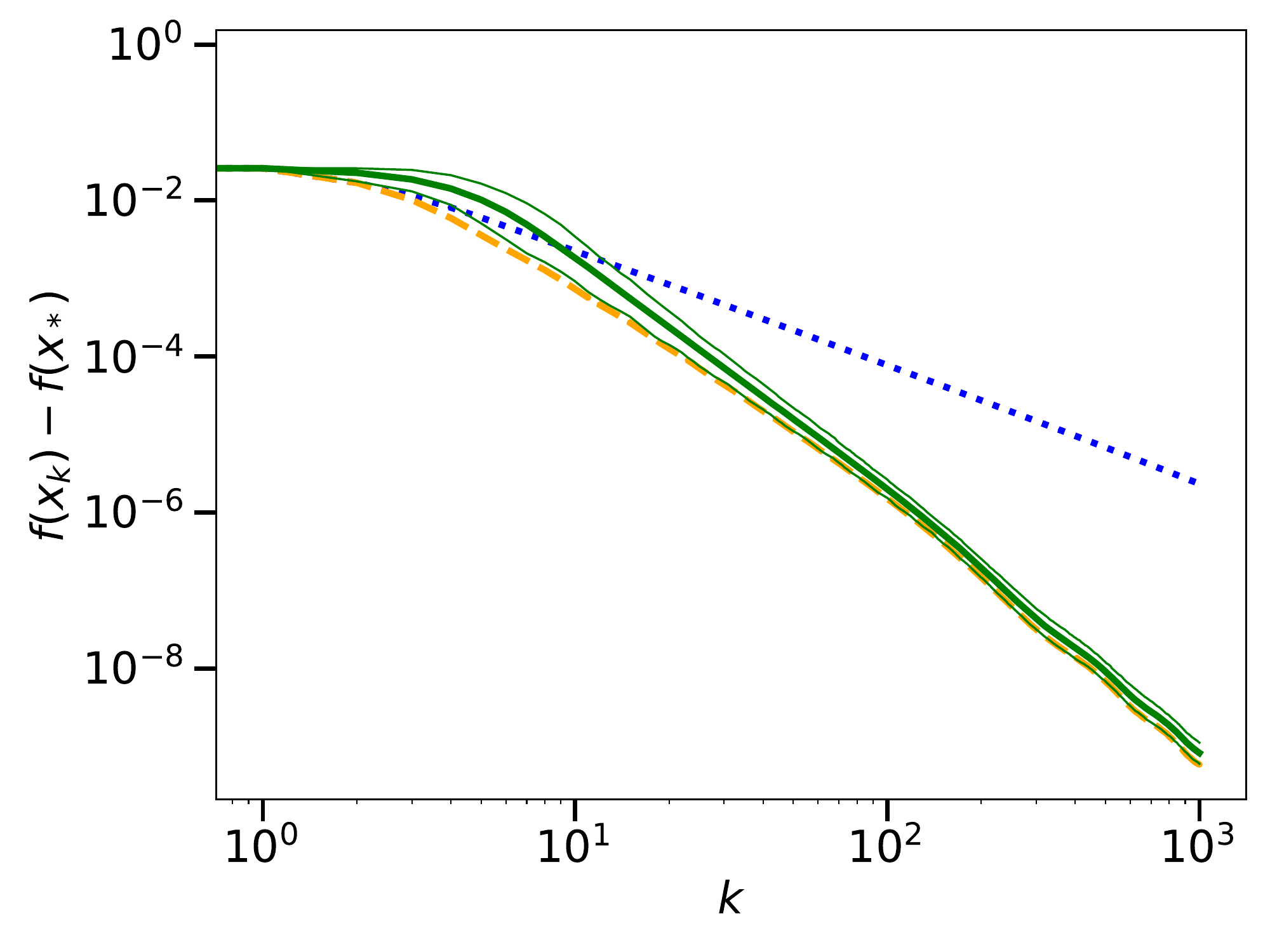}
\includegraphics[width=0.49\linewidth]{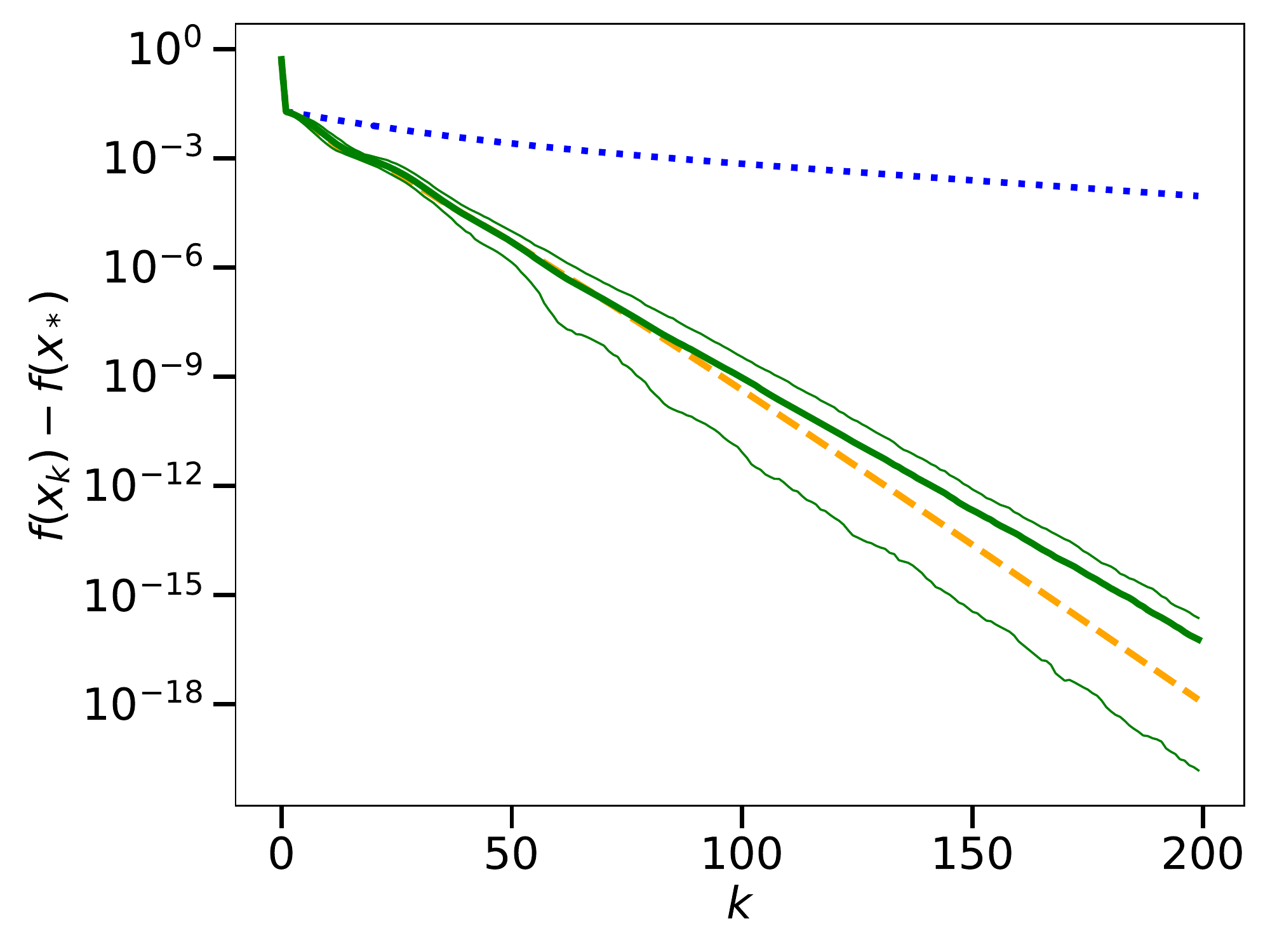}
	\caption{Comparison between gradient descent, Nesterov acceleration, and the continuized version of Nesterov acceleration, on a convex function (left plots) and a strongly convex function (right plots). For the continuized acceleration, which is randomized, the results shown in the above plots correspond to a single run. In the plots below, the thick line represents the average performance over $N = 1000$ runs of the continuized acceleration, while the thin lines represent the $5\%$ and $95\%$ quantiles.}
	\label{fig:comparison}
\end{figure}


In Figure \ref{fig:comparison}, we compare this continuized Nesterov acceleration \eqref{eq:discretization-1}-\eqref{eq:discretization-3} with the classical Nesterov acceleration \eqref{eq:nest-1}-\eqref{eq:nest-3} and gradient descent. In the strongly convex case (right), we run the algorithms with the parameters of Theorem \ref{thm:nesterov}.(\ref{it:nest-str-cvx}) and \ref{thm:discretization}.(\ref{it:str-cvx}) on the function 
\begin{equation*}
f(x_1, x_2, x_3) = \frac{\mu}{2}(x_1-1)^2 + \frac{3\mu}{2}(x_2-1)^2 + \frac{L}{2}(x_3-1)^2 \, , 
\end{equation*} with $\mu = 10^{-2}$ and $L = 1$. In the convex case, we run the algorithms with the parameters of Theorem \ref{thm:nesterov}.(\ref{it:nest-cvx}) and \ref{thm:discretization}.(\ref{it:cvx}) on the function 
\begin{equation*}
f(x_1, \dots, x_{100})= \frac{1}{2}\sum_{i = 1}^{100} \frac{1}{i^2}\left(x_i - \frac{1}{i}\right)^2 \, , 
\end{equation*}
which has negligible strong convexity parameter. All iterations were initialized from $x_0 = z_0 = 0$.  
\clearpage
\subsection{Simulation of Accelerated Randomized Gossip\label{app:simulation_gossip}}

\begin{figure}[h]
\centering
\subfigure[Line graph, 30 nodes]{
    \includegraphics[scale=0.45]{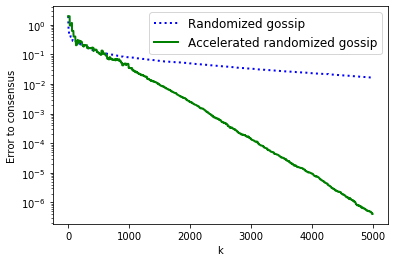}
    \label{fig:line1}
}
\subfigure[Line graph, 30 nodes]{
    \includegraphics[scale=0.45]{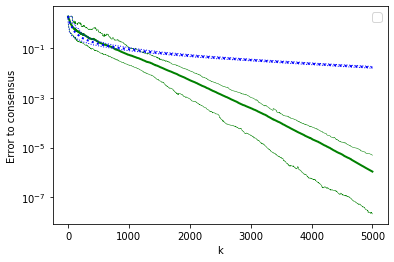}
}
\subfigure[2D-Grid, 225 nodes]{
    \includegraphics[scale=0.45]{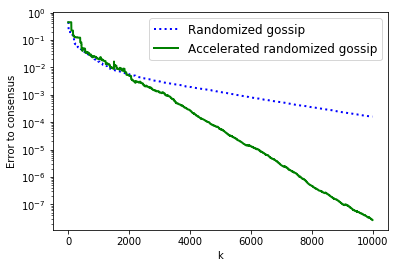}
    \label{fig:grid1}
}
\subfigure[2D-Grid, 225 nodes]{
    \includegraphics[scale=0.45]{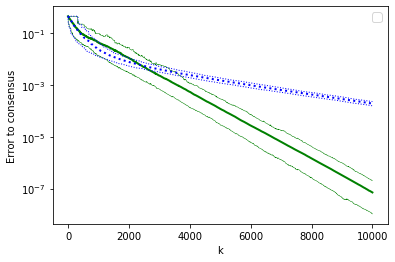}
}
\subfigure[Complete graph, 10 nodes]{
    \includegraphics[scale=0.45]{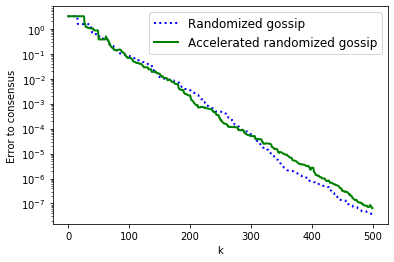}
}
\subfigure[Complete graph, 10 nodes]{
    \includegraphics[scale=0.45]{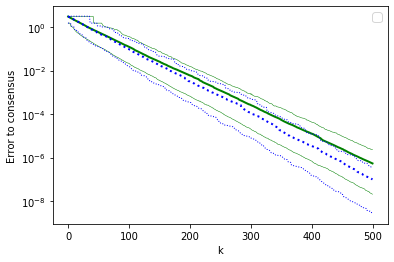}
}
\caption{Comparison between randomized gossip \citep{boyd2006randomized} and accelerated randomized gossip from Section \ref{sec:gossip}, on 3 different graphs: line with 30 nodes, 2D-Grid with 225 nodes and complete graph with 30 nodes. The probability $\cP$ on the set of edges that determines at every activation which edge is activated is uniform in all cases. Parameters of the algorithm are taken as in Theorem \ref{thm:gossip}. In all simulations, initialization was taken with a vector $x_0$ such that $x_0(v)=0$ at all nodes, except one where $x_0(v)=1$. Figures on the left represent one run of the algorithms. Figures on the right represent the average performance (thick line) for $N=1000$ runes with the same settings, and the $5\%$ and $95\%$ quantiles (thin lines). As expected, we observe acceleration one the line and the grid, but no such phenomenon on the complete graph. }\label{fig:gossip}
\end{figure}

\section{Robustness of the continuized Nesterov acceleration to additive noise}
\label{ap:additive}

In this section, we study the continuized acceleration \eqref{eq:continuized-sgd-additive-1}-\eqref{eq:continuized-sgd-additive-2} under stochastic gradients. We assume that our gradient estimates are unbiased, i.e.,
\begin{equation}
\label{eq:unbiased}
\forall x \in \R^d \, , \qquad \E_\xi \nabla f(x, \xi) = \nabla f(x) \, ,
\end{equation}
and has a uniformly bounded variance, i.e., there exists $\sigma^2 \geq 0$ such that 
\begin{equation}
\label{eq:bounded-variance}
\forall x \in \R^d \, , \qquad \E_\xi \left\Vert \nabla f(x, \xi) - \nabla f(x) \right\Vert^2 \leq \sigma^2 \, .
\end{equation}
These assumptions typically hold in the additive noise model, where $\nabla f(x,\xi) = \nabla f(x) + \xi$, and $\xi \in \R^d$ satisfies $\E\xi = 0$, $\E \Vert \xi \Vert^2 \leq \sigma^2$. By an abuse of terminology, we say that our stochastic gradients have ``additive noise'' when \eqref{eq:unbiased} and \eqref{eq:bounded-variance} hold. 

We should emphasize that similar studies of Nesterov acceleration under additive noise has been done \citep{lan2012optimal,hu2009accelerated,xia2010dual,devolder2011stochastic,cohen2018acceleration,aybat2020robust}. 

\begin{theorem}[Continuized acceleration with additive noise]
	\label{thm:additive-noise}
	Assume that the stochastic gradients are unbiased \eqref{eq:unbiased} and have a variance uniformly bounded by $\sigma^2$ \eqref{eq:bounded-variance}. Then the continuized acceleration \eqref{eq:continuized-sgd-additive-1}-\eqref{eq:continuized-sgd-additive-2} satisfies the following. 
		\begin{enumerate}
		\item For the parameters of Theorem \ref{thm:continuized}.(\ref{it:cvx}),
		\begin{align*}
\E f(x_t) - f(x_*) \leq \frac{2L\Vert z_0 -x_* \Vert^2}{t^2} + \sigma^2 \frac{t}{3L} \, .
\end{align*}
		\item Assume further that $f$ is $\mu$-strongly convex, $\mu > 0$. For the parameters of Theorem \ref{thm:continuized}.(\ref{it:str-cvx}), 
		\begin{align*}
		\E f(x_t) - f(x_*) \leq \left(f(x_0) - f(x_*) + \frac{\mu}{2} \Vert z_0 - x_* \Vert^2\right) \exp \left(-\sqrt{\frac{\mu}{L}}t\right) + \sigma^2 \frac{1}{\sqrt{\mu L}} \, .
		\end{align*}	
	\end{enumerate}
\end{theorem}
This theorem is proved in Appendix \ref{ap:proof-additive}.

In the above bounds, $L$ is a parameter of the algorithm, that can be taken greater than the best known smoothness constant of the function $f$. Increasing $L$ reduces the stepsizes of the algorithm and performs some variance reduction. If the bound $\sigma^2$ on the variance is known, one can choose $L$ optimizing the above bounds in order to obtain algorithms that adapt to additive noise. 

\begin{figure}
	\includegraphics[width=0.49\linewidth]{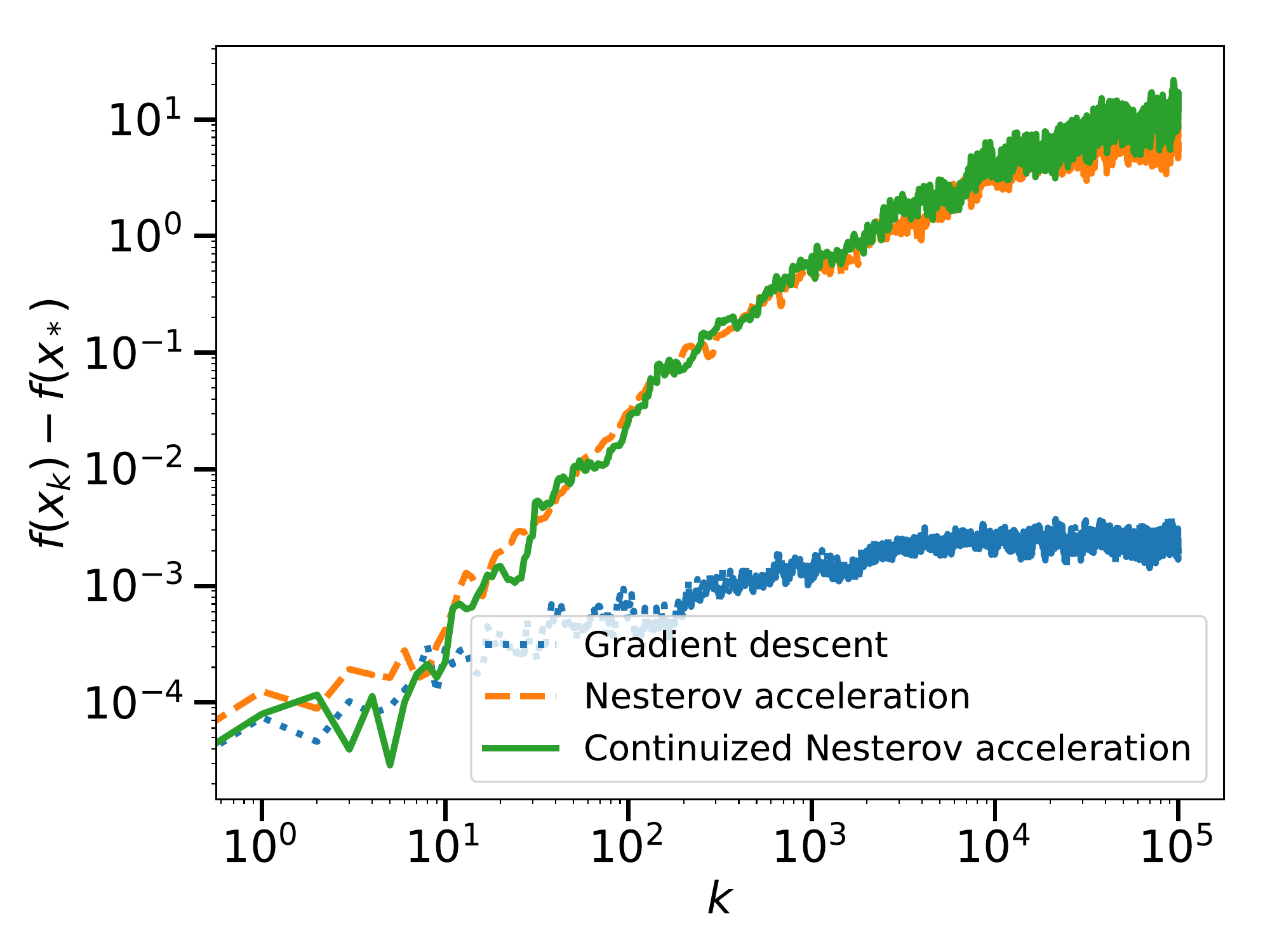}
	\includegraphics[width=0.49\linewidth]{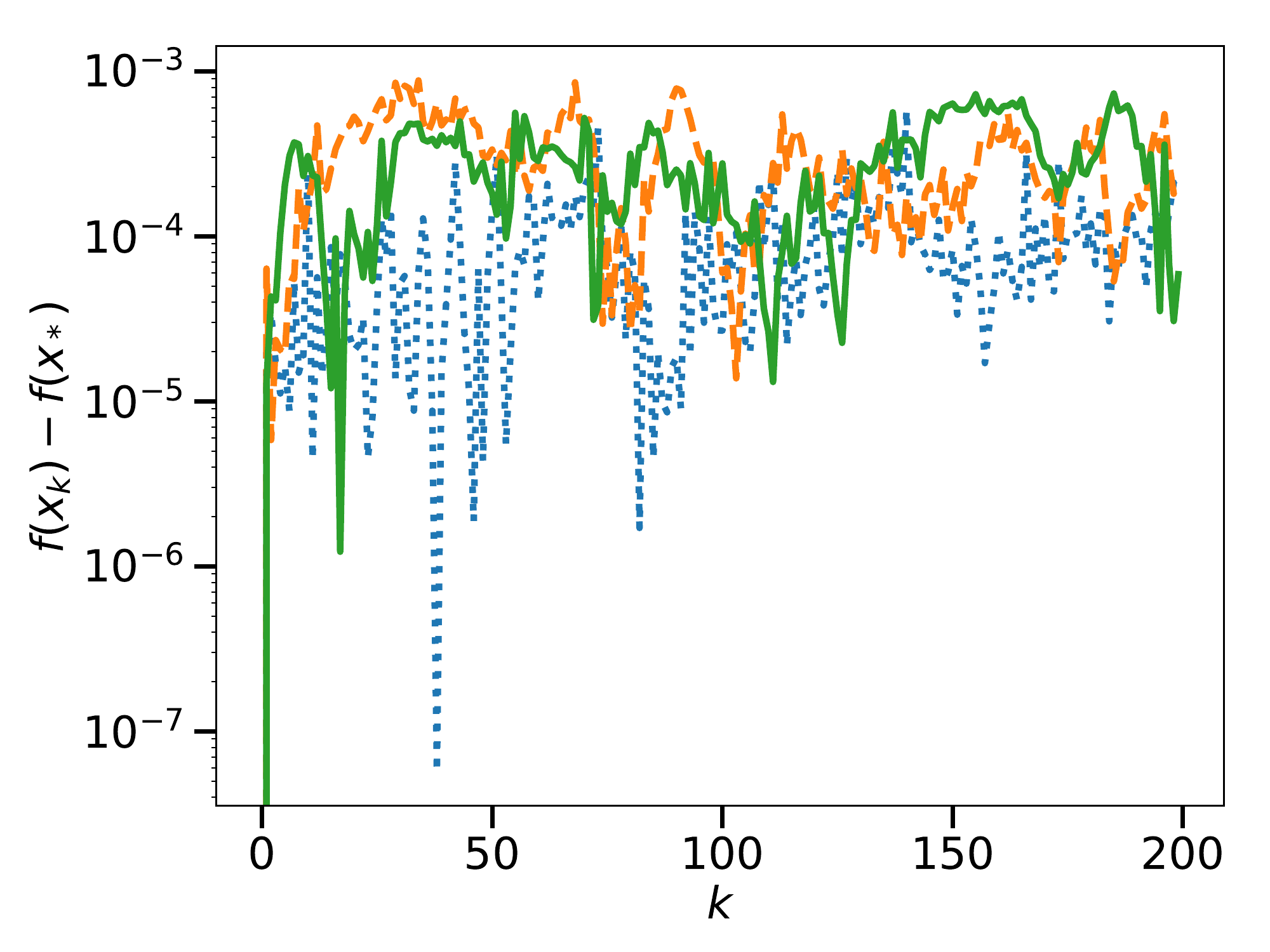}
	\includegraphics[width=0.49\linewidth]{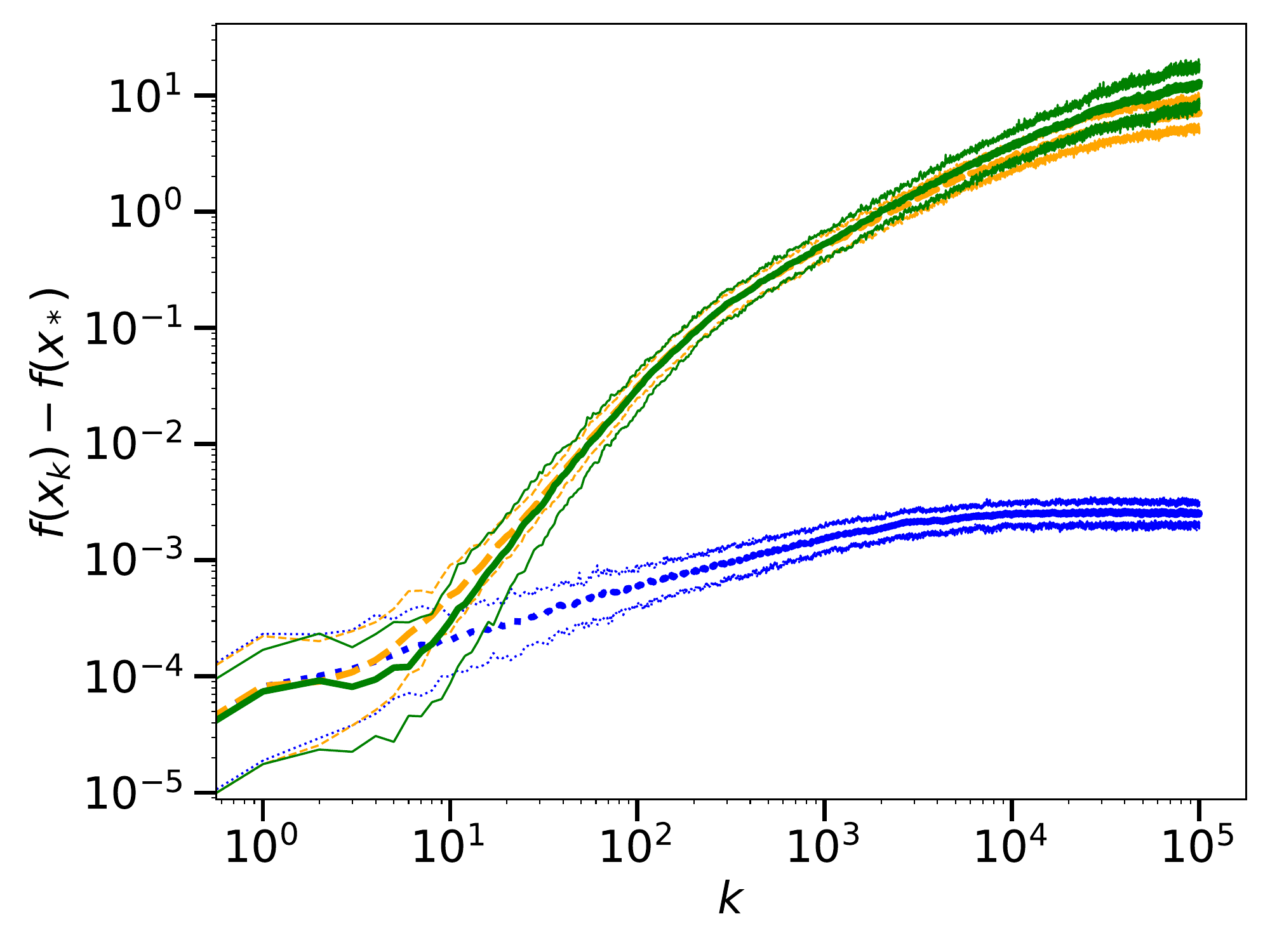}
\includegraphics[width=0.49\linewidth]{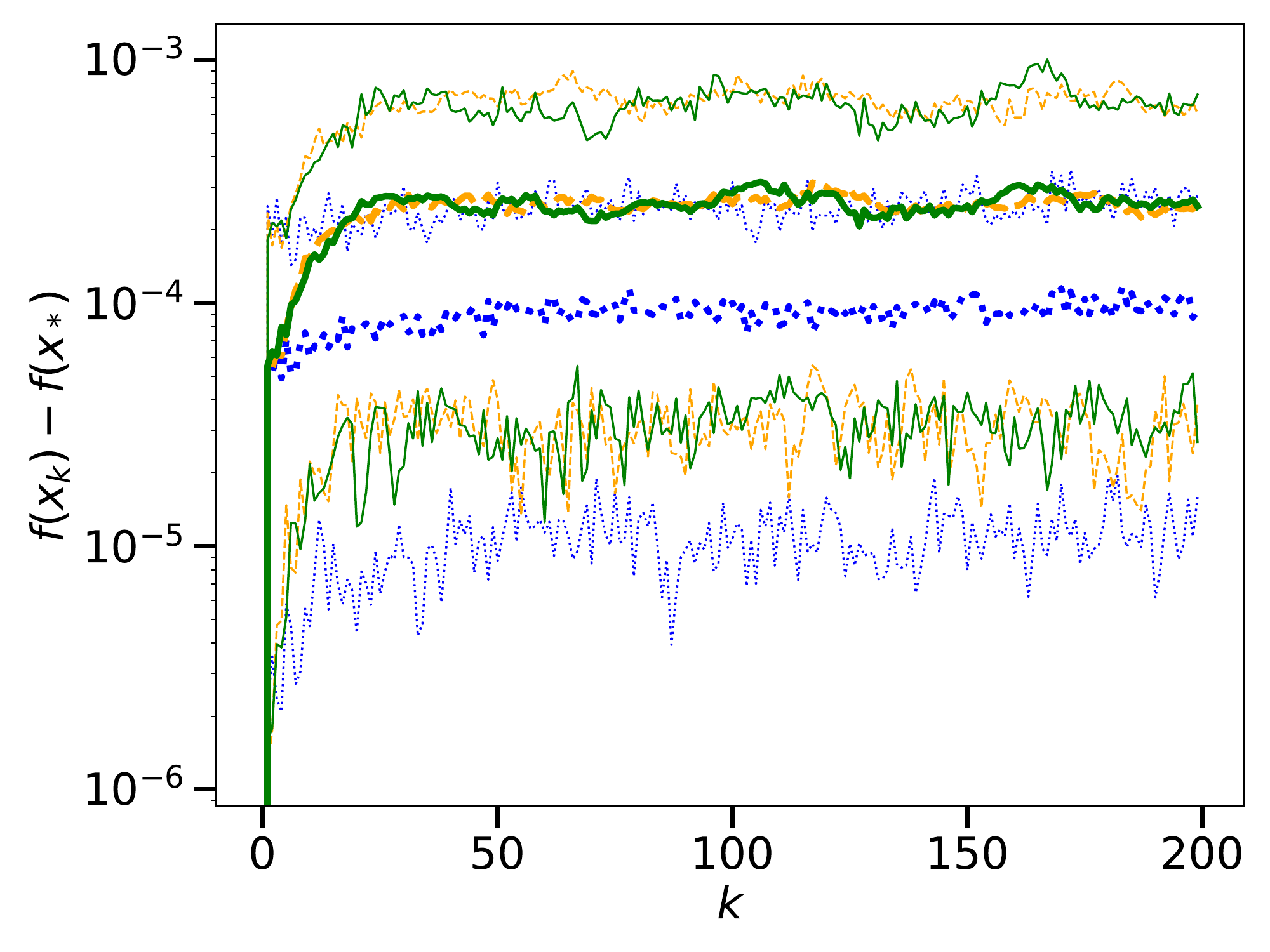}
	\caption{Effect of additive noise on gradient descent, Nesterov acceleration, and the continuized version of Nesterov acceleration, on a convex function (left) and a strongly convex function (right). All algorithms are started from the optimum $x_*$. The results shown in the above plots correspond to a single run. In the plots below, the thick line represents the average performance over $N = 100$ runs of each algorithm, while the thin lines represent the $5\%$ and $95\%$ quantiles.}
	\label{fig:comparison-add}
\end{figure}
In Figure \ref{fig:comparison-add}, we run the same simulations as in Figure \ref{fig:comparison}, with two differences: (1) we add isotropic Gaussian noise on the gradients, with covariance $10^{-4}\Id$, and (2) we initialized algorithms at the optimum, i.e., $x_0 = z_0 = x_*$. Initializing at the optimum enables to isolate the effect of the additive noise only. These simulations confirm Theorem \ref{thm:additive-noise}: the noise term is (sub-)linearly increasing in the convex case and constant in the strongly convex case. 

Note that similarly to Theorem \ref{thm:discretization}, one could obtain convergence bounds for the discrete implementation under the presence of additive noise. 

\section{Stochastic calculus toolbox}
\label{ap:toolbox}

In this appendix, we give a short introduction to the mathematical tools that we use in this paper. For more details, the reader can consult the more rigorous monographs of \citet{jacod2013limit,ikeda2014stochastic,le2016brownian}.

\subsection{Poisson point measures}

We fix $\cP$ a probability law on some space $\Xi$. 

\begin{definition}
	A \emph{(homogenous) Poisson point measure} on $\R_{\geq 0} \times \Xi$, with intensity $\nu(\diff t, \diff \xi) = \diff t \otimes \diff \cP(\xi)$, is a random measure $N$ on $\R_{\geq 0} \times \Xi$ such that
	\begin{itemize}
		\item For any disjoint measurable subsets $A$ and $B$ of $\R_{\geq 0} \times \Xi$, $N(A)$ and $N(B)$ are independent.
		\item For any measurable subset $A$ of $\R_{\geq 0} \times \Xi$, $N(A)$ is a Poisson random variable with parameter $\nu(A)$. (If $\nu(A) = \infty$, $N(A)$ is equal to $\infty$ almost surely.)  
	\end{itemize}
\end{definition}

\begin{proposition}
	\label{prop:decomposition-poisson-measure}
	Let $N$ be a Poisson point measure on $\R_{\geq 0} \times \Xi$ with intensity $\diff t \otimes \diff \cP(\xi)$. 
	
	There exists a decomposition $\diff N(t,\xi) = \sum_{k\geq 1} \delta_{(T_k,\xi_k)}(\diff t, \diff \xi)$ on $\R_{\geq 0} \times \Xi$ where $0 < T_1 < T_2 < T_3 < \dots$ and $\xi_1, \xi_2, \xi_3, \dots \in \Xi$ satisfy:
	\begin{itemize}
		\item $T_1, T_2-T_1, T_3 - T_2, \dots$ are i.i.d.~of law exponential with rate $1$,
		\item $\xi_1, \xi_2, \xi_3, \dots$ are i.i.d.~of law $\cP$ and independent of the $T_1, T_2, T_3, \dots$. 
	\end{itemize}
\end{proposition}

\begin{definition}
	Let $N$ be a Poisson point measure on $\R_{\geq 0} \times \Xi$ with intensity $\diff t \otimes \diff \cP(\xi)$. The \emph{filtration $\cF_t$, $t \geq 0$, generated by $N$} is defined by the formula
	\begin{align*}
		\cF_t = \sigma\left( N([0,s]\times A) \, , \, s\leq t, A \subset \Xi \text{ measurable}\right) \, .
	\end{align*}
\end{definition}

\subsection{Martingales and supermartingales}

Let $(\Omega, \cF, \P)$ be a probability space and $\cF_t$, $t \geq 0$, a filtration on this probability space.

\begin{definition}
A random process $x_t \in \R^d$, $t \geq 0$, is \emph{adapted} if for all $t \geq 0$, $x_t$ is $\cF_t$-measurable. 
An adapted process $x_t \in \R$, $t \geq 0$ is a \emph{martingale} (resp.~\emph{supermartingale}) if for all $0 \leq s \leq t$, $\E[x_t | \cF_s] = x_s$ (resp.~$\E[x_t | \cF_s] \leq x_s$).
\end{definition}

\begin{definition}
	A random variable $T \in [0, \infty]$ is a \emph{stopping time} if for all $t \geq 0$, $\{T \leq t\} \in \cF_t$. 
\end{definition}

\begin{definition}
	\label{def:cadlag}
	A function $x_t, t \geq 0$, is said to be \emph{c\`adl\`ag} if it is right continuous and for every $t > 0$, the limit $x_{t-} := \lim_{s \to t, s < t} x_s$ exists and is finite. 
\end{definition}

\begin{theorem}[Martingale stopping theorem]
	\label{thm:stopping}
	Let $x_t$, $ t \geq 0$, be a martingale (resp.~supermartingale) with c\`adl\`ag trajectories and uniformly integrable. Let $T$ be a stopping time. Then $\E X_T = X_0$ (resp.~$\E X_T \leq X_0$).
\end{theorem}

\subsection{Stochastic ordinary differential equation with Poisson jumps}

The continuized processes are the composition of an ordinary differential equation and stochastic Poisson jumps. It is thus a piecewise-deterministic Markov process \cite{davis1984piecewise,davis2018markov}, a special case of stochastic models that do not include any diffusion term. The stochastic calculus of this class of processes is particularly intuitive: there is no Ito correction term as with diffusive processes. 

We fix $\cP$ a probability law on some space $\Xi$, $N$ a Poisson point measure on $\R_{\geq 0} \times \Xi$ with intensity $\diff t \otimes \diff \cP(\xi)$, and denote $\cF_t$, $t \geq 0$, the filtration generated by $N$. 

\begin{definition}
	Let $b: \R^d \to \R^d$ and $G: \R^d \times \Xi \to \R^d$ be two functions. An random process $x_t \in \R^d$, $t \geq 0$, is said to be a solution of the equation
	\begin{align*}
		\diff x_t = b(x_t) \diff t + \int_{\Xi} G(x_t, \xi) \diff N(t, \xi)
	\end{align*}
	if it is adapted, c\`adl\`ag, and for all $t \geq 0$,
		\begin{align*}
		x_t =  x_0 + \int_{0}^{t}b(x_s) \diff s + \int_{[0,t]\times \Xi} G(x_{s-}, \xi) \diff N(s, \xi) \, .
	\end{align*}
If we consider the decomposition $\diff N(t,\xi) = \sum_{k\geq 1} \delta_{(T_k,\xi_k)}(\diff t, \diff \xi)$ given by Proposition \ref{prop:decomposition-poisson-measure}, then 
\begin{align*}
	\int_{[0,t]\times \Xi} G(x_{s-}, \xi) \diff N(s, \xi) = \sum_{k \geq 1} \bfone_{\{T_k \leq t\}} G(x_{T_k-}, \xi_k) \, .
\end{align*}
\end{definition}

Here, we consider only autonomous equations as $b$ and $G$ are a function of $x_t$, but not of $t$. However, there is no loss of generality, one can study time-dependent systems by studying the equation in the variable $(t,x_t)$. This trick is used in Appendix \ref{ap:analysis-continuized}.

\begin{proposition}
	\label{prop:ito}
	 Let $x_t \in \R^d$ be a solution of
	 	\begin{align*}
	 	\diff x_t = b(x_t) \diff t + \int_{\Xi} G(x_t, \xi) \diff N(t, \xi)
	 \end{align*}
	and $\varphi : \R^d \to \R$ be a smooth function. Then 
	\begin{align*}
		\varphi(x_t) = \varphi(x_0) + \int_{0}^{t} \langle \nabla \varphi(x_s), b(x_s) \rangle \diff s + \int_{[0,t]\times \Xi}\left(\varphi(x_{s-}+G(x_{s-}, \xi) ) - \varphi(x_{s-}) \right) \diff N(s, \xi) \, . 
	\end{align*}
Moreover, we have the decomposition 
\begin{align*}
	&\int_{[0,t]\times \Xi} \left(\varphi(x_{s-}+G(x_{s-}, \xi) ) - \varphi(x_{s-}) \right) \diff N(s, \xi)  \\
	&\qquad = 	\int_{0}^{t}\int_{\Xi} \left(\varphi(x_{s}+G(x_{s}, \xi) ) - \varphi(x_{s} )\right) \diff t \diff \cP(\xi) + M_t \, , 
\end{align*}
where $M_t = \int_{[0,t]\times \Xi} \left(\varphi(x_{s-}+G(x_{s-}, \xi) ) - \varphi(x_{s-}) \right) (\diff N(s, \xi) - \diff t \diff \cP(\xi))$ is a martingale. 
\end{proposition}

This proposition is an elementary calculus of variations formula: to compute the value of the observable $\varphi(x_t)$, one must sum the effects of the continuous part and of the Poisson jumps. Moreover, the integral with respect to the Poisson measure $N$ becomes a martingale if the same integral with respect to its intensity measure $\diff t \otimes \diff \cP(\xi)$ is removed.

\section{Analysis of the continuized Nesterov acceleration}
\label{ap:analysis-continuized}

To encompass the proofs in the convex and in the strongly convex cases in a unified way, we assume $f$ is $\mu$-strongly convex, $\mu \geq 0$. If $\mu > 0$, this corresponds to assuming the $\mu$-strong convexity in the usual sense; if $\mu = 0$, it means that we only assume the function to be convex. In other words, the proofs in the convex case can be obtained by taking $\mu = 0$ below.   

In this section, $\cF_t$, $t\geq 0$, is the filtration associated to the Poisson point measure $N$.

\subsection{Sketch of proof for Theorem \ref{thm:continuized}}
\label{ap:sketch-prrof}

A complete and rigorous proof is given in Appendix \ref{ap:proof-continuized}. Here, we only provide the heuristic of the main lines of the proof. 
	
	The proof is similar to the one of Nesterov acceleration: we prove that for some choices of parameters $\eta_t, \eta_t', \gamma_t, \gamma_t'$, $t \geq 0$, and for some functions $A_t, B_t$, $t \geq 0$, 
	\begin{equation*}
		\phi_t = A_t\left(f(x_t)-f(x_*)\right) + \frac{B_t}{2}\Vert z_t - x_* \Vert^2  
	\end{equation*}
	is a supermartingale. In particular, this implies that $\E\phi_t$ is a Lyapunov function, i.e., a non-increasing function of $t$. 
	
	To prove that $\phi_t$ is a supermartingale, it is sufficient to prove that for all infinitesimal time intervals $[t,t+\diff t]$, $\E_t \phi_{t+\diff t} \leq \phi_t$, where $\E_t$ denotes the conditional expectation knowing all the past of the Poisson process up to time $t$. Thus we would like to compute the first order variation of $\E_t \phi_{t+\diff t}$. This implies computing the first order variation of $\E_t f(x_{t+\diff t})$. 
	
	From \eqref{eq:continuized-1}, we see that $f(x_t)$ evolves for two reasons between $t$ and $t + \diff t$:
	\begin{itemize}
		\item $x_t$ follows the linear ODE \eqref{eq:continuous-part-1}, which results in the infinitesimal variation $f(x_t) \rightarrow f(x_t) + \eta_t \langle \nabla f(x_t), z_t-x_t \rangle \diff t$, and
		\item with probability $\diff t$, $x_t$ takes a gradient step, which results in a macroscopic variation $f(x_t) \rightarrow f\left(x_t - \gamma_t \nabla f(x_t)\right)$. 
	\end{itemize} 
Combining both variations, we obtain that 
\begin{equation*}
\E_tf(x_{t+\diff t}) \approx f(x_t) + \eta_t \langle \nabla f(x_t), z_t-x_t \rangle \diff t + \diff t \left(f\left(x_t - \gamma_t \nabla f(x_t)\right) - f(x_t)\right) \, ,
\end{equation*}
where the $\diff t$ in the second term corresponds to the probability that a gradient step happens; note that the latter event is independent of the past up to time $t$. 

A similar computation can be done for $\E_t \Vert z_t - x_* \Vert^2$. Putting things together, we obtain 
\begin{align*}
    \E_t \phi_{t+\diff t} - \phi_t \approx \diff t \bigg(&\frac{\diff A_t}{\diff t}(f(x_t)-f(x_*))+A_t\eta_t\langle \nabla f(x_t), z_t-x_t\rangle \\
    &- A_t \left(f(x_t-\gamma_t \nabla f(x_t))-f(x_t)\right) + \frac{\diff B_t}{\diff t} \frac{1}{2}\Vert z_t - x_* \Vert^2 \\
    &+ B_t \eta_t' \langle z_t - x_*, x_t - z_t \rangle + \frac{B_t}{2} \big(\Vert z_t - \gamma_t'\nabla f(x_t)-x_* \Vert^2 - \Vert z_t - x_* \Vert^2 \big)\bigg) \, .
\end{align*}
Using convexity and strong convexity inequalities, and a few computations, we obtain the following upper bound: 
\begin{align*}
    \E_t \phi_{t+\diff t} - \phi_t \lesssim\diff t \bigg(&\left(\frac{\diff A_t}{\diff t} - A_t \eta_t\right)  \langle \nabla f(x_t), x_t- x_* \rangle + \left(\frac{\diff B_t }{\diff t}- B_t\eta_t'\right) \frac{1}{2}\Vert z_t - x_* \Vert^2 \\
&+ (A_t \eta_t- B_t \gamma_t')  \langle \nabla f(x_t), z_t- x_* \rangle + \left(B_t \eta_t'- \frac{\diff A_t}{\diff t} \mu\right) \frac{1}{2}\Vert x_t - x_* \Vert^2\\ 
&  +\left( B_t \gamma_t'^2 -A_t \gamma_t \left(2 - {L\gamma_t}\right)\right)  \frac{1}{2}\Vert \nabla f(x_t) \Vert^2 \bigg) \, .
\end{align*}
We want this infinitesimal variation to be non-positive. Here, we choose the parameters so that $\gamma_t = 1/L$, and all prefactors in the above expression are zero. This gives some constraints on the choices of parameters. We show that only one degree of freedom is left: the choice of the function $A_t$, that must satisfy the ODE 
\begin{equation*}
\frac{\diff^2}{\diff t^2} \left(\sqrt{A_t}\right) = \frac{\mu}{4L} \sqrt{A_t} \, ,
\end{equation*}
but whose initialization remains free. Once the initialization of the function $A_t$ is chosen, this determines the full function $A_t$ and, through the constraints, all parameters of the algorithm. As $\phi_t$ is a supermartingale (by design), a bound on the performance of the algorithm is given by
\begin{align*}
\E f(x_t) - f(x_*) \leq \frac{\E \phi_t}{A_t} \leq \frac{\phi_0}{A_t} \, .
\end{align*}
The results presented in Theorem \ref{thm:continuized} correspond to one special choice of initialization for the function~$A_t$. 

	In this sketch of proof, our derivation of the infinitesimal variation is intuitive and elementary; however it can be made more rigorous and concise---albeit more technical---using classical results from stochastic calculus, namely Proposition \ref{prop:ito}. This is our approach in Appendix \ref{ap:proof-continuized}.

\subsection{Noiseless case: proofs of Theorem \ref{thm:continuized} and of the bounds of Theorem \ref{thm:discretization}}
\label{ap:proof-continuized}

In this section, we analyze the convergence of the continuized iteration \eqref{eq:continuized-1}-\eqref{eq:continuized-2}, that we recall for the reader's convenience:
\begin{align*}
\diff x_{t} &= \eta_t (z_t-x_t) \diff t - \gamma_t \nabla f(x_t)\diff N(t)  \, ,\\
\diff z_{t} &= \eta_t' (x_t-z_t)\diff t - \gamma_t'\nabla f(x_t)\diff N(t) \, .
\end{align*}
The choices of parameters $\eta_t, \eta'_t, \gamma_t, \gamma'_t$, $t\geq 0$, and the corresponding convergence bounds follow naturally from the analysis. We seek sufficient conditions under which the function
\begin{equation*}
\phi_t = A_t\left(f(x_t)-f_*\right) + \frac{B_t}{2} \Vert z_t - x_* \Vert^2  
\end{equation*}
is a supermartingale. 

The process $\bar{x}_t = (t,x_t,z_t)$ satisfies the equation 
\begin{align*}
	&\diff \bar{x}_t = b(\bar{x}_t) \diff t + G(\bar{x}_t) \diff N(t) \, , && b(\bar{x}_t) = \begin{pmatrix}
		1 \\
		\eta_t(z_t-x_t) \\
		\eta_t'(x_t-z_t) 
	\end{pmatrix} \, , &&G(\bar{x}_t) = \begin{pmatrix}
	0 \\
	-\gamma_t \nabla f(x_t) \\
	-\gamma_t' \nabla f(x_t)
	\end{pmatrix} \, .
\end{align*}
We thus apply Proposition \ref{prop:ito} to $\phi_t = \varphi(\bar{x}_t) = \varphi(t,x_t,z_t)$ where 
\begin{align*}
\varphi(t,x,z) = A_t\left(f(x)-f(x_*)\right) + \frac{B_t}{2}\Vert z - x_* \Vert^2 \, ,
\end{align*}
we obtain:
\begin{align*}
\phi_t = \phi_0 + \int_{0}^{t} \langle \nabla \varphi(\bar{x}_s), b(\bar{x}_s) \rangle \diff s + \int_{0}^{t} \left(\varphi(\bar{x}_{s}+G(\bar{x}_{s}) ) - \varphi(\bar{x}_{s}) \right) \diff s + M_t \, ,
\end{align*}
where $M_t$ is a martingale. Thus, to show that $\varphi_t$ is a supermartingale, it is sufficient to show that the map $t \mapsto \int_{0}^{t} \langle \nabla \varphi(\bar{x}_s), b(\bar{x}_s) \rangle \diff s + \int_{0}^{t} \left(\varphi(\bar{x}_{s}+G(\bar{x}_{s}) ) - \varphi(\bar{x}_{s}) )\right) \diff s$ is non-increasing almost surely, i.e., 
\begin{align*}
I_t := \langle \nabla \varphi(\bar{x}_t), b(\bar{x}_t) \rangle +\varphi(\bar{x}_{t}+G(\bar{x}_{t}) ) - \varphi(\bar{x}_{t})  \leq 0 \, . 
\end{align*}
We now compute 
\begin{align*}
	\langle \nabla \varphi(\bar{x}_{t}), b(\bar{x}_{t}) \rangle &= \partial_t \varphi(\bar{x}_{t}) + \langle \partial_x \varphi(\bar{x}_{t}), \eta_t(z_t-x_t) \rangle + \langle \partial_z \varphi(\bar{x}_{t}), \eta_t'(x_t-z_t) \rangle \\
	&= \frac{\diff A_t}{\diff t} \left(f(x_t)-f(x_*)\right) + \frac{\diff B_t}{\diff t} \frac{1}{2}\Vert z_t - x_* \Vert^2 + A_t \eta_t \langle \nabla f(x_t), z_t - x_t \rangle \\
	&\qquad + B_t \eta_t' \langle z_t-x_*, x_t-z_t \rangle \, .
\end{align*}
Here, we use that as $f$ is $\mu$-strongly convex, 
\begin{equation*}
f(x_t) - f(x_*) \leq \langle \nabla f(x_t), x_t- x_* \rangle - \frac{\mu}{2} \Vert x_t - x_* \Vert^2 \, , 
\end{equation*}
and the simple bound
\begin{align*}
\langle z_t-x_*, x_t-z_t \rangle &= \langle z_t-x_*, x_t-x_* \rangle - \Vert z_t - x_* \Vert^2 \leq \Vert z_t - x_* \Vert \Vert x_t - x_* \Vert - \Vert z_t - x_* \Vert^2 \\
&\leq \frac{1}{2} \left(\Vert z_t - x_* \Vert^2  + \Vert x_t - x_* \Vert^2 \right) - \Vert z_t - x_* \Vert^2 = \frac{1}{2} \left(\Vert x_t - x_* \Vert^2 -\Vert z_t - x_* \Vert^2 \right) \, .
\end{align*}
This gives
\begin{align}
\langle \nabla \varphi(\bar{x}_{t}), b(\bar{x}_{t}) \rangle &\leq
\left(\frac{\diff A_t}{\diff t} - A_t \eta_t\right)  \langle \nabla f(x_t), x_t- x_* \rangle + \left(B_t \eta_t'- \frac{\diff A_t}{\diff t} \mu\right) \frac{1}{2}  \Vert x_t - x_* \Vert^2 \label{eq:aux-4}\\&+ \left(\frac{\diff B_t }{\diff t}- B_t\eta_t'\right) \frac{1}{2}\Vert z_t - x_* \Vert^2
+ A_t \eta_t \langle \nabla f(x_t), z_t - x_* \rangle \, . \label{eq:aux-5}
\end{align}
Further, 
\begin{align*}
\varphi(\bar{x}_{t}+G(\bar{x}_{t}) ) - \varphi(\bar{x}_{t}) &= A_t\left(f(x_t-\gamma_t \nabla f(x_t))-f(x_t)\right) \\
&\qquad+ \frac{B_t}{2} \left(\Vert (z_t-x_*) - \gamma_t' \nabla f(x_t) \Vert^2 - \Vert z_t-x_*\Vert^2\right) \, .
\end{align*}
As $f$ is $L$-smooth, 
\begin{align*}
f(x_t-\gamma_t \nabla f(x_t))-f(x_t) &\leq \langle \nabla f(x_t) , - \gamma_t \nabla f(x_t) \rangle + \frac{L}{2} \Vert \gamma_t \nabla f(x_t) \Vert^2 \\
&= -\gamma_t \left(2 - {L\gamma_t}\right) \frac{1}{2}\Vert \nabla f(x_t) \Vert^2 \, .
\end{align*}
This gives
\begin{align}
\varphi(\bar{x}_{t}+G(\bar{x}_{t}) ) - \varphi(\bar{x}_{t}) &\leq
\left( B_t \gamma_t'^2 -A_t \gamma_t \left(2 - {L\gamma_t}\right)\right)  \frac{1}{2}\Vert \nabla f(x_t) \Vert^2 - B_t \gamma_t' \langle \nabla f(x_t), z_t- x_* \rangle  \, . \label{eq:aux-6}
\end{align}
Finally, combining \eqref{eq:aux-4}-\eqref{eq:aux-5} with \eqref{eq:aux-6}, we obtain
\begin{align}
I_t  &\leq \left(\frac{\diff A_t}{\diff t} - A_t \eta_t\right)  \langle \nabla f(x_t), x_t- x_* \rangle + \left(\frac{\diff B_t }{\diff t}- B_t\eta_t'\right) \frac{1}{2} \Vert z_t - x_* \Vert^2 \label{eq:aux-7}\\
&\qquad+ (A_t \eta_t- B_t \gamma_t')  \langle \nabla f(x_t), z_t- x_* \rangle + \left(B_t \eta_t'- \frac{\diff A_t}{\diff t} \mu\right) \frac{1}{2}\Vert x_t - x_* \Vert^2\\ 
&\qquad   +\left( B_t \gamma_t'^2 -A_t \gamma_t \left(2 - L\gamma_t\right)\right) \frac{1}{2} \Vert \nabla f(x_t) \Vert^2 \, . \label{eq:aux-8}
\end{align}
Remember that $I_t \leq 0$ is a sufficient condition for $\phi_t$ to be a supermartingale. Here, we choose the parameters $\eta_t, \eta_t', \gamma_t, \gamma_t', t \geq 0$, so that all prefactors are $0$. We start by taking $\gamma_t \equiv \frac{1}{L}$ (other choices $\gamma_t < \frac{2}{L}$ could be possible but would give similar results) and we want to satisfy
\begin{align*}
&\frac{\diff A_t}{\diff t} = A_t \eta_t \, , &&\frac{\diff B_t }{\diff t}= B_t\eta_t' 
&&A_t \eta_t =  B_t \gamma_t' \, , &&B_t \eta_t' = \frac{\diff A_t }{\diff t} \mu \, , 
&&B_t \gamma_t'^2 = \frac{A_t}{L} \, .
\end{align*}
To satisfy the last equation, we choose 
\begin{equation}
\label{eq:formula-gamma_t'}
\gamma_t' = \sqrt{\frac{A_t}{LB_t}} \, . 
\end{equation}
To satisfy the third equation, we choose
\begin{equation}
\label{eq:formula-eta_t}
\eta_t = \frac{ B_t \gamma_t'}{A_t} = \sqrt{\frac{2B_t}{LA_t }} \, . 
\end{equation}
To satisfy the fourth equation, we choose
\begin{equation}
\label{eq:formula-eta_t'}
\eta_t' = \frac{\diff A_t}{\diff t} \frac{\mu}{ B_t} = \frac{A_t \eta_t \mu}{ B_t} = \mu \sqrt{\frac{A_t}{L B_t}} \, . 
\end{equation}
Having now all parameters $\eta_t, \eta_t', \gamma_t, \gamma_t'$ constrained, we now have that $\phi_t$ is Lyapunov if
\begin{align*}
&\frac{\diff A_t}{\diff t} = A_t \eta_t = \sqrt{\frac{A_t B_t}{L}} \, , &&\frac{\diff B_t}{\diff t} = B_t \eta'_t = \mu \sqrt{\frac{A_t B_t}{L}} \, .
\end{align*}
This only leaves the choice of the initialization $(A_0, B_0)$ as free: both the algorithm and the Lyapunov depend on it. (Actually, only the relative value $A_0/B_0$ matters.) Instead of solving the above system of two coupled non-linear ODEs, it is convenient to turn them into a single second-order linear ODE:
\begin{align}
\label{eq:coupled-first-order}
	&\frac{\diff}{\diff t}\left(\sqrt{A_t}\right) = \frac{1}{2\sqrt{A_t}} \frac{\diff A_t}{\diff t} = \frac{1}{2}\sqrt{\frac{B_t}{L}} \, , &&\frac{\diff}{\diff t}\left(\sqrt{B_t}\right) = \frac{1}{2\sqrt{B_t}} \frac{\diff B_t}{\diff t} = \frac{\mu}{2} \sqrt{\frac{A_t}{L}} \, . 
\end{align}
This can also be restated as
\begin{align}
\label{eq:second-order}
&\frac{\diff^2}{\diff t^2} \left(\sqrt{A_t}\right) = \frac{\mu}{4L} \sqrt{A_t} \, , &&\sqrt{B_t} = 2\sqrt{L} \frac{\diff }{\diff t}\left(\sqrt{A_t}\right) \, .
\end{align}

\subsubsection{Proof of the first part (convex case)}

We now assume $\mu = 0$, and we choose the solution such that $A_0 = 
0$ and $B_0 = 1$. From \eqref{eq:coupled-first-order}, we have $\frac{\diff}{\diff t}\left(\sqrt{B_t}\right) = 0$, thus $B_t \equiv 1$, and $\frac{\diff}{\diff t}\left(\sqrt{A_t}\right) = \frac{1}{2\sqrt{L}}$, thus $\sqrt{A_t} = \frac{t}{2\sqrt{L}}$. 
The parameters of the algorithm are given by \eqref{eq:formula-gamma_t'}-\eqref{eq:formula-eta_t'}: $\eta_t = \frac{2}{t}$, $\eta_t' = 0$, $\gamma'_t = \frac{t}{2L}$ (and we had chosen $\gamma_t = \frac{1}{L}$).

From the fact that $\phi_t$ is a supermartingale, we obtain that the associated algorithm satisfies 
\begin{equation*}
	\E f(x_t) - f(x_*) \leq \frac{\E \phi_t}{A_t} \leq \frac{\phi_0}{A_t} = \frac{2L\Vert z_0 - x_* \Vert^2}{t^2} \, .
\end{equation*} 
This proves the first part of Theorem \ref{thm:continuized}. 

Further, one can apply martingale stopping Theorem \ref{thm:stopping} to the supermartingale $\phi_t$ with the stopping time $T_k$ to obtain 
\begin{align*}
	\E \left[ A_{T_k} \left(f(\tilde{x}_k) - f(x_*) \right) \right]  = \E \left[ A_{T_k} \left( f(x_{T_k}) - f(x_*) \right) \right] \leq \E \phi_{T_k} \leq \phi_0 = \Vert z_0 - x_* \Vert^2 \, . 
\end{align*}
This proves the formula of Theorem \ref{thm:discretization}.\ref{it:cvx}.

\subsubsection{Proof of the second part (strongly convex case)}

We now assume $\mu > 0$. We consider the solution of \eqref{eq:second-order} that is exponential:
\begin{align*}
&\sqrt{A_t} = \sqrt{A_0} \exp\left(\frac{1}{2} \sqrt{\frac{\mu}{L}}t\right) \, , && \sqrt{B_t } = \sqrt{A_0} \sqrt{\mu} \exp\left(\frac{1}{2} \sqrt{\frac{\mu}{L}}t\right) \, .
\end{align*}
The parameters of the algorithm are given by \eqref{eq:formula-gamma_t'}-\eqref{eq:formula-eta_t'}: $\eta_t = \eta_t' =  \sqrt{\frac{\mu}{L}}$, $\gamma'_t = \frac{1}{\sqrt{\mu L}}$ (and we had chosen $\gamma_t = \frac{1}{L}$). 

From the fact that $\phi_t$ is a supermartingale, we obtain that the associated algorithm satisfies 
\begin{align*}
\E f(x_t) - f(x_*) \leq \frac{\E \phi_t}{A_t} \leq \frac{\phi_0}{A_t} &= \frac{A_0 (f(x_0)-f(x_*)) + A_0 \frac{\mu}{2}\Vert z_0 - x_* \Vert^2}{A_t} \\&= \left(f(x_0)-f(x_*) + \frac{\mu}{2}\Vert z_0 - x_* \Vert^2\right) \exp\left(- \sqrt{\frac{\mu}{L}}t\right)\, .
\end{align*}
This proves the second part of Theorem \ref{thm:continuized}. Similarly to above, one can also apply the martingale stopping theorem to prove the formula of Theorem \ref{thm:discretization}.\ref{it:str-cvx}. 

\begin{remark}
	In the above derivation, in both the convex and strongly convex cases, we choose a particular solution of \eqref{eq:second-order}, while several solutions are possible. In the convex case, we make the choice $A_0 = 0$ to have a succinct bound that does not depend on $f(x_0) - f(x_*)$. More importantly, in the strongly convex case, we choose the solution that satisfies the relation $\sqrt{\mu} \sqrt{A_t} = \sqrt{B_t}$, which implies that $\eta_t, \eta_t', \gamma_t'$, are constant functions of $t$, and $\eta_t = \eta_t'$. These conditions help solving in closed form the continuous part of the process 
	\begin{align*}
	&\diff x_t = \eta_t (z_t - x_t) \diff t \, ,  \\
	&\diff z_t = \eta_t'(x_t - z_t) \diff t \, ,
	\end{align*}
	which is crucial if we want to have a discrete implementation of our method (for more details, see Theorem \ref{thm:discretization} and its proof). However, in the strongly convex case, considering other solutions would be interesting, for instance to have an algorithm converging to the convex one as $\mu \to 0$. 
	\end{remark}

\subsection{With additive noise: proof of Theorem \ref{thm:additive-noise}}
\label{ap:proof-additive}

The proof of this theorem is along the same lines as the proof of Theorem \ref{thm:continuized} above. Here, we only give the major differences. 

We analyze the convergence of the continuized stochastic iteration \eqref{eq:continuized-sgd-additive-1}-\eqref{eq:continuized-sgd-additive-2}, that we recall for the reader's convenience:
\begin{align*}
\diff x_t = \eta_t (z_t - x_t) \diff t - \gamma_t \int_{\Xi}\nabla f(x_t, \xi) \diff N(t,\xi) \, , \\
\diff z_t = \eta_t' (x_t - z_t) \diff t- \gamma_t'  \int_{\Xi}\nabla f(x_t, \xi) \diff N(t,\xi) \, . 
\end{align*}
In this setting, we loose the property that
\begin{equation*}
\phi_t = A_t\left(f(x_t)-f_*\right) + \frac{B_t}{2}\Vert z_t - x_* \Vert^2  
\end{equation*}
is a supermartingale. However, we bound the increase of $\phi_t$.

The process $\bar{x}_t = (t,x_t,z_t)$ satisfies the equation 
\begin{align*}
&\diff \bar{x}_t = b(\bar{x}_t) \diff t + \int_{\Xi}G(\bar{x}_t,\xi) \diff N(t,\xi), && b(\bar{x}_t) = \begin{pmatrix}
1 \\
\eta_t(z_t-x_t) \\
\eta_t'(x_t-z_t) 
\end{pmatrix}, &&G(\bar{x}_t,\xi) = \begin{pmatrix}
0 \\
-\gamma_t \nabla f(x_t,\xi) \\
-\gamma_t' \nabla f(x_t,\xi)
\end{pmatrix}.
\end{align*}
We apply Proposition \ref{prop:ito} to $\phi_t = \varphi(\bar{x}_t) = \varphi(t,x_t,z_t)$
and obtain
\begin{align}
\label{eq:decomposition-phi-additive}
\phi_t = \phi_0 + \int_{0}^{t} I_s \diff s + M_t \, ,
\end{align}
where $M_t$ is a martingale and
\begin{align*}
I_t = \langle \nabla \varphi(\bar{x}_t), b(\bar{x}_t) \rangle +\E_\xi\varphi(\bar{x}_{t}+G(\bar{x}_{t}, \xi) ) - \varphi(\bar{x}_{t}) \, . 
\end{align*}
The computation of the first term remains the same: the inequality \eqref{eq:aux-4}-\eqref{eq:aux-5} holds. The computation of the second term becomes 
\begin{align*}
\E_\xi \varphi(\bar{x}_{t}+G(\bar{x}_{t},\xi) ) - \varphi(\bar{x}_{t}) &= A_t\left(\E_\xi f(x_t-\gamma_t \nabla f(x_t,\xi))-f(x_t)\right) \\
&\qquad+ \frac{B_t}{2} \left(\E_\xi\Vert (z_t-x_*) - \gamma_t' \nabla f(x_t,\xi) \Vert^2 - \Vert z_t-x_*\Vert^2\right) \, .
\end{align*}
As $f$ is $L$-smooth, 
\begin{align*}
f(x_t-\gamma_t \nabla f(x_t,\xi))-f(x_t) &\leq \langle \nabla f(x_t) , - \gamma_t \nabla f(x_t,\xi) \rangle + \frac{L}{2} \Vert \gamma_t \nabla f(x_t,\xi) \Vert^2  \, , \\
\E_\xi f(x_t-\gamma_t \nabla f(x_t,\xi))-f(x_t) &\leq \langle \nabla f(x_t) , - \gamma_t \E_\xi\nabla f(x_t,\xi) \rangle + \frac{L}{2} \E_\xi\Vert \gamma_t \nabla f(x_t,\xi) \Vert^2  \, .
\end{align*}
By assumptions \eqref{eq:unbiased} and \eqref{eq:bounded-variance}, the stochastic gradient $\nabla f(x,\xi)$ is unbiased and has a variance bounded by $\sigma^2$, which implies $\E_\xi\Vert  \nabla f(x_t,\xi) \Vert^2 \leq \Vert   \nabla f(x_t) \Vert^2 + \sigma^2$. Thus 
\begin{align*}
\E_\xi f(x_t-\gamma_t \nabla f(x_t, \xi ))-f(x_t) 
&\leq -\gamma_t \left(2 - L\gamma_t\right)  \frac{1}{2}\Vert \nabla f(x_t) \Vert^2 + \sigma^2 \frac{L\gamma_t^2}{2 } \, .
\end{align*}
Similarly, 
\begin{align*}
	\E_\xi\Vert (z_t-x_*) - \gamma_t' \nabla f(x_t,\xi) \Vert^2 - \Vert z_t-x_*\Vert^2 &= -2 \gamma_t' \langle \E_\xi \nabla f(x_t,\xi), z_t-x_* \rangle + \gamma_t'^2 \E_\xi \Vert  \nabla f(x_t,\xi) \Vert^2  \\
	&\leq -2 \gamma_t' \langle  \nabla f(x_t), z_t-x_* \rangle + \gamma_t'^2  \Vert  \nabla f(x_t) \Vert^2 + \sigma^2 \gamma_t'^2 \, .
\end{align*}
This gives
\begin{align*}
\varphi(\bar{x}_{t}+G(\bar{x}_{t}) ) - \varphi(\bar{x}_{t}) &\leq
\left( B_t \gamma_t'^2 -A_t \gamma_t \left(2 - L\gamma_t\right)\right) \frac{1}{2} \Vert \nabla f(x_t) \Vert^2 - B_t \gamma_t' \langle \nabla f(x_t), z_t- x_* \rangle \\
&\qquad + \frac{\sigma^2}{2} \left(A_t L\gamma_t^2 + B_t \gamma_t'^2\right) \, .
\end{align*}
Combining the bounds, we obtain 
\begin{align*}
I_t  &\leq \left(\frac{\diff A_t}{\diff t} - A_t \eta_t\right)  \langle \nabla f(x_t), x_t- x_* \rangle + \left(\frac{\diff B_t }{\diff t}- B_t\eta_t'\right) \frac{1}{2}\Vert z_t - x_* \Vert^2 \\
&\qquad+ (A_t \eta_t- B_t \gamma_t')  \langle \nabla f(x_t), z_t- x_* \rangle + \left(B_t \eta_t'- \frac{\diff A_t}{\diff t} \mu\right) \frac{1}{2}\Vert x_t - x_* \Vert^2\\ 
&\qquad   +\left( B_t \gamma_t'^2 -A_t \gamma_t \left(2 - L\gamma_t\right)\right) \frac{1}{2} \Vert \nabla f(x_t) \Vert^2 + \frac{\sigma^2}{2} \left(A_t L \gamma_t^2 + B_t \gamma_t'^2 \right)\, ,
\end{align*}
which is an additive perturbation of the bound \eqref{eq:aux-7}-\eqref{eq:aux-8} in the noiseless case, with a perturbation proportional to $\sigma^2$. The choices of parameters of Theorem \ref{thm:continuized} cancel all first five prefactors, and satisfy $\gamma_t = \frac{1}{L}$, $A_t L \gamma_t^2 = B_t \gamma_t'^2$. We thus obtain 
\begin{align*}
I_t \leq \sigma^2 \frac{A_t}{L} \, . 
\end{align*}
This bound controls the increase of $\phi_t$. Using the decomposition \eqref{eq:decomposition-phi-additive}, we obtain
\begin{align*}
\E f(x_t) - f(x_*) &\leq \frac{\E \phi_t}{A_t} \leq \frac{\phi_0}{A_t} + \frac{\int_{0}^{t}\E I_s \diff s}{A_t}  \\
&\leq \frac{A_0 (f(x_0)-f(x_*)) + B_0 \Vert z_0 - x_* \Vert^2}{A_t} + \frac{\sigma^2}{L}\frac{\int_{0}^{t} A_s \diff s}{A_t} \, .
\end{align*} 

\subsubsection{Proof of the first part (convex case)}

In this case, $A_t = \frac{t^2}{2L}$ and $B_0 = 1$. Thus $\int_{0}^{t} A_s \diff s = \frac{1}{2L} \frac{t^3}{3}$. Thus
		\begin{align*}
\E f(x_t) - f(x_*) \leq \frac{2L\Vert z_0 -x_* \Vert^2}{t^2} + \sigma^2 \frac{t}{3L} \, .
\end{align*} 

\subsubsection{Proof of the second part (strongly convex case)}

In this case, $A_t = A_0 \exp\left(\sqrt{\frac{\mu}{L}}t\right)$ and $B_0 = A_0 \frac{\mu}{2}$. Thus $\int_{0}^{t} A_s \diff s \leq A_0 \sqrt{\frac{\mu}{L}}^{-1}\exp\left(\sqrt{\frac{\mu}{L}}t\right) = \sqrt{\frac{L}{\mu}} A_t$. Thus 
\begin{align*}
\E f(x_t) - f(x_*) \leq \left(f(x_0) - f(x_*) + \frac{\mu}{2} \Vert z_0 - x_* \Vert^2\right) \exp \left(-\sqrt{\frac{\mu}{L}}t\right) + \sigma^2 \frac{1}{\sqrt{\mu L}} \, .
\end{align*}

\subsection{With Pure Multiplicative Noise: Proof of Theorem \ref{thm:multiplicative-noise}}

The proof of this theorem mimics the proof of Theorem \ref{thm:continuized}, with a slightly different Lyapunov function. 

We recall that in Section \ref{sec:stochastic}, the function $f$ is of the form:
\begin{equation*}
    \forall x\in\R^d, f(x)=\esp{\frac{1}{2}(\langle a,x\rangle -b)^2},
\end{equation*}
where $\xi=(a,b)\in\R^d\times \R$ is of law $\cP$. Thanks to the \emph{noiseless assumption}, for $H=\esp{aa^{\top}}$, we also have:
\begin{equation*}
    \forall x\in\R^d, f(x)=\frac{1}{2}\NRM{x-x_*}_H^2.
\end{equation*}
The Lyapunov function studied in the proof of Theorem \ref{thm:continuized} would then write as, for $t\in\R_{\geq0}$:
\begin{equation*}
    \phi_t=\frac{A_t}{2}\NRM{x_t-x_*}_H^2 +\frac{B_t}{2}\NRM{z_t-x_*}^2.
\end{equation*}
An acceleration of stochastic gradient descent using this Lyapunov function has been done by \citet{vaswani2019fast}. In order to have an analysis similar to Nesterov acceleration, the authors make a strong growth condition, which is too strong for many stochastic gradient problems and for our application to gossip algorithms. Instead, our analysis requires a bounded statistical condition number~$\Tilde{\kappa}$, and performs a shift in terms of dependency over $H$: $\NRM{x-x_*}_H^2$ becomes $\NRM{x-x_*}^2$, and $\NRM{z_t-x_*}^2$ becomes $\NRM{z_t-x_*}_{H^{-1}}^2$. The new Lyapunov function writes:
\begin{equation*}
        \phi_t=\frac{A_t}{2}\NRM{x_t-x_*}^2 +\frac{B_t}{2}\NRM{z_t-x_*}_{H^{-1}}^2.
\end{equation*}
As in Theorem \ref{thm:continuized}, the proof consists in proving that for carefully chosen parameters, $\phi_t$ is a supermatingale.
The process $\bar{x}_t = (t,x_t,z_t)$ satisfies the equation 
\begin{align*}
&\diff \bar{x}_t = b(\bar{x}_t) \diff t + \int_{\Xi}G(\bar{x}_t,\xi) \diff N(t,\xi), && b(\bar{x}_t) = \begin{pmatrix}
1 \\
\eta_t(z_t-x_t) \\
\eta_t'(x_t-z_t) 
\end{pmatrix}, &&G(\bar{x}_t,\xi) = \begin{pmatrix}
0 \\
-\gamma_t \nabla f(x_t,\xi) \\
-\gamma_t' \nabla f(x_t,\xi)
\end{pmatrix}.
\end{align*}
We apply Proposition \ref{prop:ito} to $\phi_t = \varphi(\bar{x}_t) = \varphi(t,x_t,z_t)$
and obtain:
\begin{align*}
\phi_t = \phi_0 + \int_{0}^{t} I_s \diff s + M_t \, ,
\end{align*}
where $M_t$ is a martingale and
\begin{align*}
I_t = \langle \nabla \varphi(\bar{x}_t), b(\bar{x}_t) \rangle +\E_\xi\varphi(\bar{x}_{t}+G(\bar{x}_{t}, \xi) ) - \varphi(\bar{x}_{t}) \, . 
\end{align*}
Since the Lyapunov function is not the same, we need to explicit here each term. The first term writes:
\begin{align*}
    \langle \nabla \varphi(\bar{x}_t), b(\bar{x}_t) \rangle = & \frac{1}{2}\frac{\diff A_t}{\diff t} \NRM{x_t-x_*}^2 + \frac{1}{2}\frac{\diff B_t}{\diff t} \NRM{z_t-x_*}^2_{H^{-1}}\\
    & + A_t\eta_t\langle x_t-x_*, z_t-x_t \rangle + B_t\eta'_t \langle H^{-1} (z_t - x_*),x_t-z_t\rangle.
\end{align*}
Mimicking the proof of Theorem \ref{thm:continuized}, we write
\begin{align*}
    \frac{1}{2}\NRM{x_t-x_*}^2 \leq \NRM{x_t-x_*}^2 - \frac{\mu}{2}\NRM{x_t-x_*}_{H^{-1}}^2,
\end{align*}
and
\begin{align*}
    \langle H^{-1} (z_t - x_*),x_t-z_t\rangle &= \langle z_t-x_*,x_t-x_*\rangle_{H^{-1}} - \NRM{z_t-x_*}^2_{H^{-1}}\\
    &\leq \frac{1}{2}\big( \NRM{x_t-x_*}^2_{H^{-1}}-\NRM{z_t-x_*}^2_{H^{-1}}\big).
\end{align*}
Hence:
\begin{align*}
    \langle \nabla \varphi(\bar{x}_t), b(\bar{x}_t) \rangle &\leq \frac{\diff A_t}{\diff t}\NRM{x_t-x_*}^2 +\left(B_t\eta'_t -\frac{\diff A_t}{\diff t}\mu\right)\frac{1}{2}\NRM{x_t-x_*}_{H^{-1}}^2\\
    &\qquad+ \left(\frac{\diff B_t}{\diff t}- B_t\eta'_t\right) \frac{1}{2}\NRM{z_t-x_*}_{H^{-1}}^2 + A_t\eta_t \langle x_t-x_*,z_t-x_t\rangle \, .
\end{align*}
Further, 
\begin{align*}
\varphi(\bar{x}_{t}+G(\bar{x}_{t}) ) - \varphi(\bar{x}_{t}) &= \frac{A_t}{2}\left(\NRM{x_t-\gamma_t \nabla f(x_t,\xi)-x_*}^2-\NRM{x_t-x_*}^2\right) \\
&\qquad+ \frac{B_t}{2} \left(\Vert (z_t-x_*) - \gamma_t' \nabla f(x_t,\xi) \Vert^2_{H^{-1}} - \Vert z_t-x_*\Vert^2_{H^{-1}}\right) \, .
\end{align*}
Then, expanding and taking expectation over $\xi$ of the first term:
\begin{align*}
    \E_{\xi}\left[\frac{1}{2}\NRM{x_t-\gamma_t \nabla f(x_t,\xi)-x_*}^2-\frac{1}{2}\NRM{x_t-x_*}^2\right] &= \frac{\gamma_t^2}{2}\E_{\xi}\left[\NRM{\nabla f(x_t,\xi)}^2\right] - \gamma_t\langle H(x_t-x_*),x_t-x_*\rangle\\
    &\leq \left(\frac{R^2\gamma_t^2}{2} -\gamma_t\right)\NRM{x_t-x_*}^2_H,
\end{align*}
where we used the definition of $R^2$ in Equation \eqref{eq:R_squared}:  
\begin{align*}
    \E_{\xi}\left[\NRM{\nabla f(x_t,\xi)}^2\right]&=(x_t-x_*)^{\top}\esp{ a a^{\top} a a^{\top}}(x_t-x_*)\\
    &= (x_t-x_*)^{\top} \esp{\NRM{a}^2aa^{\top}}(x_t-x_*)\\
    &\leq R^2 (x_t-x_*)^{\top} H(x_t-x_*).
\end{align*}
The second term writes:
\begin{align*}
    \frac{1}{2}\E_{\xi}\left[\Vert (z_t-x_*) - \gamma_t' \nabla f(x_t,\xi) \Vert^2_{H^{-1}} - \Vert z_t-x_*\Vert^2_{H^{-1}}\right] &= \frac{{\gamma'_t}^2}{2}\E_{\xi}\left[ \NRM{\nabla f(x_t,\xi)}^2_{H^{-1}}\right] \\
    &\qquad- \gamma'_t\langle x_t-x_*,z_t-x_*\rangle \\
    &\leq \frac{\Tilde{\kappa}{\gamma'_t}^2}{2}\NRM{x_t-x_*}^2_{H} \\
    &\qquad- \gamma'_t\langle x_t-x_*,z_t-x_*\rangle,
\end{align*}
where we used the definition of $\Tilde{\kappa}$ in Equation \eqref{eq:kappa_tilde}:
\begin{align*}
    \E_{\xi}\left[ \NRM{\nabla f(x_t,\xi)}^2_{H^{-1}}\right]&= (x_t-x_*)^\top \esp{aa^{\top}H^{-1}aa^{\top}}(x_t-x_*)\\
    &=(x_t - x_*)^{\top}\esp{a \NRM{a}_{H^{-1}}^2a^{\top}}(x_t-x_*)\\
    &\leq \Tilde{\kappa}(x_t-x_*)^{\top}H(x_t-x_*).
\end{align*}
Combining these inequalities gives the following upper-bound on $I_t$:
\begin{align*}
    I_t  &\leq \left(\frac{\diff A_t}{\diff t}-A_t \eta_t\right)  \NRM{x_t-x_*}^2 + \left(\frac{\diff B_t }{\diff t}- B_t\eta_t'\right) \frac{1}{2} \Vert z_t - x_* \Vert^2_{H^{-1}} \\
&\qquad+  (A_t\eta_t - B_t \gamma_t')\langle x_t-x_*, z_t- x_* \rangle + \left(B_t \eta_t'- \frac{\diff A_t}{\diff t} \mu\right) \frac{1}{2}\Vert x_t - x_* \Vert^2_{H^{-1}}\\ 
&\qquad   +\left(\Tilde{\kappa} B_t \gamma_t'^2 -A_t \gamma_t \left(2 - R^2\gamma_t\right)\right) \frac{1}{2} \NRM{x_t-x_*}^2_{H}
\end{align*}
Since $I_t\leq0$ is still a sufficient condition for $\phi_t$ to be a supermartingale, we choose parameters such that all prefactors are equal to 0. We first take $\gamma_t=\frac{1}{R^2}$, and we want to satisfy:
\begin{align*}
&\frac{\diff A_t}{\diff t} = A_t \eta_t \, , &&\frac{\diff B_t }{\diff t}= B_t\eta_t' 
&&A_t \eta_t =  B_t \gamma_t' \, , &&B_t \eta_t' = \frac{\diff A_t }{\diff t} \mu \, , 
&&B_t \gamma_t'^2 = \frac{A_t}{\Tilde{\kappa}R^2} \, .
\end{align*}
To satisfy that last equality, we choose:
\begin{equation*}
    \gamma'_t=\sqrt{\frac{A_t}{B_t\Tilde{\kappa}R^2}}.
\end{equation*}
The rest of the proof then follows just as in the proof of Theorem \ref{ap:proof-continuized}.

\section{Proof of Theorem \ref{thm:discretization}}
\label{ap:proof-thm-discretization}

By integrating the ODE
\begin{align*}
	&\diff x_t = \eta_t (z_t - x_t) \diff t \, ,  \\
	&\diff z_t = \eta_t'(x_t - z_t) \diff t \, ,
\end{align*}
between $T_k$ and $T_{k+1}-$, we obtain that there exists $\tau_k, \tau_k''$, such that 
\begin{align}
	&\tilde{y}_k = x_{T_{k+1}-} = x_{T_k} + \tau_k(z_{T_k} - x_{T_k}) = \tilde{x}_k + \tau_k (\tilde{z}_k - \tilde{x}_k) \, , \label{eq:aux-1}\\
	&z_{T_{k+1}-} = z_{T_k} + \tau_k''(x_{T_k} - z_{T_k}) = \tilde{z}_k + \tau_k'' (\tilde{x}_k - \tilde{z}_k) \, . \nonumber
\end{align}
From the first equation, we have $\tilde{x}_k = \frac{1}{1 - \tau_k} \left(\tilde{y}_k - \tau_k \tilde{z}_k\right)$, which gives by substitution in the second equation,
\begin{align*}
	z_{T_{k+1}-} &= \tilde{z}_k + \tau_k'' \left(\frac{1}{1 - \tau_k} \left(\tilde{y}_k - \tau_k \tilde{z}_k\right) - \tilde{z}_k\right) \\
	&= \tilde{z}_k + \tau_k' (\tilde{y}_k - \tilde{z}_k) \, ,
\end{align*}
where $\tau_k' = \frac{\tau_k''}{1 - \tau_k}$.

Further, from \eqref{eq:gradient-1}-\eqref{eq:gradient-2}, we obtain the equations 
\begin{align}
	&\tilde{x}_{k+1} = x_{T_{k+1}} = x_{T_{k+1}-} - \gamma_{T_{k+1}} \nabla f (x_{T_{k+1}-}) = \tilde{y}_k - \gamma_{T_{k+1}} \nabla f (\tilde{y}_k) \, , \label{eq:aux-2}\\
	&\tilde{z}_{k+1} = z_{T_{k+1}} = z_{T_{k+1}-} - \gamma'_{T_{k+1}} \nabla f (x_{T_{k+1}-}) = \tilde{z}_k + \tau_k' (\tilde{y}_k - \tilde{z}_k) - \gamma'_{T_{k+1}} \nabla f (\tilde{y}_k) \, .  \label{eq:aux-3}
\end{align}
The stated equation \eqref{eq:discretization-1}-\eqref{eq:discretization-3} are the combination of \eqref{eq:aux-1}, \eqref{eq:aux-2} and \eqref{eq:aux-3}. 

\begin{enumerate}
	\item The parameters of Theorem \ref{thm:continuized}.(\ref{it:cvx}) are $\eta_t = \frac{2}{t}, \eta_t' = 0, \gamma_t = \frac{1}{L}$ and $\gamma_t' = \frac{t}{2L}$. In this case, the ODE 
	\begin{align*}
		&\diff x_t = \eta_t (z_t - x_t) \diff t = \frac{2}{t} (z_t - x_t) \diff t \, ,  \\
		&\diff z_t = \eta_t'(x_t - z_t) \diff t = 0\, ,
	\end{align*}
	can be integrated in closed form: for $t \geq t_0$,
	\begin{align*}
		&x_t = z_{t_0} + \left(\frac{t_0}{t}\right)^2 (x_{t_0}-z_{t_0})= x_{t_0} + \left(1 - \left(\frac{t_0}{t}\right)^2\right)(z_{t_0}-x_{t_0}) \, ,  \\
		&z_t = z_{t_0}\, .
	\end{align*} 
	In particular, taking $t_0 = T_k$, $t = T_{k+1}-$, we obtain $\tau_k = 1 - \left(\frac{T_k}{T_{k+1}}\right)^2$, $\tau_k'' = 0$ and thus $\tau_k' = \frac{\tau_k''}{1 - \tau_k} = 0$. Finally, $\tilde{\gamma}_k = \gamma_{T_k} = \frac{1}{L}$ and $\tilde{\gamma}'_k = \gamma_{T_k}' = \frac{T_k}{2L}$. 
	\item The parameters of Theorem \ref{thm:continuized}.(\ref{it:str-cvx}) are $\eta_t = \eta_t' \equiv \sqrt{\frac{\mu}{L}}, \gamma_t \equiv \frac{1}{L}$ and $\gamma_t' \equiv \frac{1}{\sqrt{\mu L }}$. In this case, the ODE 
	\begin{align*}
		&\diff x_t = \eta_t (z_t - x_t) \diff t = \sqrt{\frac{\mu}{L}} (z_t - x_t) \diff t \, ,  \\
		&\diff z_t = \eta_t'(x_t - z_t) \diff t = \sqrt{\frac{\mu}{L}} (x_t - z_t) \diff t \, ,
	\end{align*}
	can also be integrated in closed form: for $t \geq t_0$,
	\begin{align*}
		x_t &= \frac{x_{t_0}+z_{t_0}}{2} + \frac{x_{t_0}-z_{t_0}}{2} \exp\left(-2\sqrt{\frac{\mu}{L}}(t-t_0)\right) \\ 
		&= x_{t_0} + \frac{1}{2} \left(1 - \exp\left(-2\sqrt{\frac{\mu}{L}}(t-t_0)\right)\right)(z_{t_0} - x_{t_0}) \, , \\
		z_t &= \frac{x_{t_0}+z_{t_0}}{2} + \frac{z_{t_0}-x_{t_0}}{2} \exp\left(-2\sqrt{\frac{\mu}{L}}(t-t_0)\right) \\
		&= z_{t_0} + \frac{1}{2} \left(1 - \exp\left(-2\sqrt{\frac{\mu}{L}}(t-t_0)\right)\right)(x_{t_0} - z_{t_0}) \, . 
	\end{align*}
	In particular, taking $t_0 = T_k$, $t = T_{k+1}-$, we obtain $\tau_k = \tau_k'' = \frac{1}{2} \left(1 - \exp\left(-2\sqrt{\frac{\mu}{L}}(T_{k+1}-T_k)\right)\right)$ and thus $\tau_k' = \frac{\tau_k''}{1 - \tau_k} = \tanh\left(\sqrt{\frac{\mu}{L}}(T_{k+1}-T_k)\right)$. Finally, $\tilde{\gamma}_k = \gamma_{T_k} = \frac{1}{L}$ and $\tilde{\gamma}'_k = \gamma_{T_k}' = \frac{1}{\sqrt{\mu L}}$.
\end{enumerate}

\section{Heuristic ODE scaling limit of the continuized acceleration}
\label{ap:scaling-limit}

\subsection{Convex case}

With the choices of parameters of Theorem \ref{thm:continuized}.(\ref{it:cvx}), the continuized acceleration is 
\begin{align*}
\diff x_{t} &= \frac{2}{t} (z_t-x_t) \diff t - \frac{1}{L} \nabla f(x_t)\diff N(t)  \, ,\\
\diff z_{t} &=  - \frac{t}{2L}\nabla f(x_t)\diff N(t) \, .
\end{align*}
The ODE scaling limit is obtained by taking the limit $L \to \infty$ (so that the stepsize $1/L$ vanishes) and rescaling the time $s = t/\sqrt{L}$. Some law of large number argument heuristically gives us that, as $L \to \infty$, $\diff N(t) = \diff N(\sqrt{L}s) \approx \sqrt{L} \diff s$. Thus in the limit, we obtain  
\begin{align*}
\diff x_{s} &= \frac{2}{\sqrt{L}s} (z_s-x_s) \sqrt{L}\diff s - \frac{1}{L} \nabla f(x_s)\sqrt{L} \diff s  \, ,\\
\diff z_{s} &=  - \frac{\sqrt{L}s}{2L}\nabla f(x_s)\sqrt{L}\diff s  \, .
\end{align*}
The second term of the first equation becomes negligible in the limit. Thus the equations simplify to 
\begin{align*}
\frac{\diff x_{s}}{\diff s} &= \frac{2}{s} (z_s-x_s)  \, ,\\
\frac{\diff z_{s}}{\diff s} &=  - \frac{s}{2}\nabla f(x_s) \, .
\end{align*}
Thus 
\begin{align*}
- \frac{s}{2}\nabla f(x_s) = \frac{\diff z_{s}}{\diff s} = \frac{\diff }{\diff s} \left(x_s + \frac{s}{2}\frac{\diff x_{s}}{\diff s} \right) = \frac{\diff x_s }{\diff s} + \frac{1}{2}\frac{\diff x_s }{\diff s}  + \frac{s}{2}\frac{\diff^2 x_s }{\diff s^2} \, ,
\end{align*}
and thus 
\begin{align*}
\frac{\diff^2 x_s }{\diff s^2} + \frac{3}{s}\frac{\diff x_s }{\diff s} + \nabla f(x_s) = 0 \, .
\end{align*}
This is the same limiting ODE as the one found by \citet{su2014differential} for Nesterov acceleration. 

\subsection{Strongly-convex case}

With the choices of parameters of Theorem \ref{thm:continuized}.(\ref{it:str-cvx}), the continuized acceleration is 
\begin{align*}
\diff x_{t} &= \sqrt{\frac{\mu}{L}} (z_t-x_t) \diff t - \frac{1}{L} \nabla f(x_t)\diff N(t)  \, ,\\
\diff z_{t} &=  \sqrt{\frac{\mu}{L}}(x_t-z_t)\diff t - \frac{1}{\sqrt{\mu L }}\nabla f(x_t)\diff N(t) \, .
\end{align*}
Again, we take joint scaling $L \to \infty$, $s = t/\sqrt{L}$, with the approximation $\diff N(t) \approx \sqrt{L} \diff s$. We obtain 
\begin{align*}
\diff x_{s} &= \sqrt{\frac{\mu}{L}} (z_s-x_s) \sqrt{L} \diff s - \frac{1}{L} \nabla f(x_s)\sqrt{L} \diff s  \, ,\\
\diff z_{s} &=  \sqrt{\frac{\mu}{L}}(x_s-z_s)\sqrt{L} \diff s - \frac{1}{\sqrt{\mu L }}\nabla f(x_s)\sqrt{L} \diff s \, .
\end{align*}
As before, the second term of the first equation becomes negligible in the limit. Thus the equations simplify to 
\begin{align}
\frac{\diff x_{s}}{\diff s} &= \sqrt{\mu} (z_s-x_s)  \, , \label{eq:aux-10}\\
\frac{\diff z_{s}}{\diff s} &=  \sqrt{\mu}(x_s-z_s)- \frac{1}{\sqrt{\mu }}\nabla f(x_s) \label{eq:aux-11}\, .
\end{align}
From \eqref{eq:aux-10}, we have $z_s = x_s + \frac{1}{\sqrt{\mu}} \frac{\diff x_{s}}{\diff s}$, and by substitution in \eqref{eq:aux-11}, we obtain 
\begin{align*}
\frac{\diff^2 x_s }{\diff s^2} + 2 \sqrt{\mu}\frac{\diff x_s }{\diff s} + \nabla f(x_s) = 0 \, .
\end{align*}
This is the so-called ``low-resolution'' ODE for Nesterov acceleration of \citet{shi2018understanding}.

\section{Continuized Accelerated Coordinate Descent with arbitrary sampling}
\label{app:coordinate_descent}
In this section, we focus on the following problem:
\begin{equation}
    \min_{x \in \R^d} f(x),
\end{equation}
where $f$ is of the form $f: x\mapsto g(\proj x)$ for some function $g$ and projector $\proj  \in \R^{d \times d}$ (such that $\proj ^2 = \proj $). We further assume that $f$ is smooth with respect to some matrix $M \in \R^{d\times d}$ and $\mu$-strongly convex with respect to $\proj $, \emph{i.e.}:
\begin{align*}
    \frac{\mu}{2}\|x - y\|^2_\proj  \leq f(x) - f(y) - \nabla f(x)^\top (x - y) \leq \frac{1}{2}\|x - y\|^2_M. 
\end{align*}
Note that $\mu$ can be equal to zero, but convergence will be slower in this case.  We analyze the convergence of the following continuized coordinate descent iteration:
\begin{align}
\begin{split}\label{eq:continuized_cd}
\diff x_t = \eta_t (z_t - x_t) \diff t - \gamma_t \int_{\Xi} \frac{\proj_{\xi \xi}}{\cP_\xi}\nabla f(x_t, \xi) \diff N(t,\xi) \, , \\
\diff z_t = \eta_t' (x_t - z_t) \diff t- \gamma_t'  \int_{\Xi}\nabla f(x_t, \xi) \diff N(t,\xi) \,,
\end{split}
\end{align}
where
\begin{equation} \label{eq:coordinate_gradient}
    \nabla f(x_t, \xi) = \frac{1}{\cP_\xi} \nabla_\xi f(x_t),
\end{equation}
with the coordinate gradient $\nabla_\xi f(x_t) = e_\xi e_\xi^\top \nabla f(x_t)$, with $e_\xi \in \R^d$ the unit vector associated with coordinate $\xi \in \{1, \dots, d\}$ and $\cP_\xi$ and $\diff N$ are defined as in Section~\ref{sec:gossip}. Note that these iterations are slightly different from the previous stochastic gradient iteration since the stochastic gradient is not the same for $x_t$ and $z_t$ (same direction but different magnitudes). The following theorem is a continuized version of~\citet{hendrikx2018accelerated}, which is itself largely based on~\citet{neststich2017acdm}.

\begin{theorem}[Continuized acceleration of coordinate descent]
	\label{thm:coord_descent}
	Assume that the stochastic gradients are of the coordinate descent form~\eqref{eq:coordinate_gradient}. Besides, choose parameter $L$ such that: 
	\begin{equation}
	L \geq \max_{\xi \in \Xi} \frac{M_{\xi\xi} \proj_{\xi \xi}}{\cP_\xi^2} \, .    
	\end{equation}
	Then the continuized acceleration~\eqref{eq:continuized_cd} satisfies the following:
		\begin{enumerate}
		\item For $\eta_t = \frac{2}{t}, \eta_t' = 0, \gamma_t = \frac{1}{L}, \gamma_t' = \frac{t}{2L}$,
		\begin{align*}
\E f(x_t) - f(x_*) \leq \frac{2L\Vert z_0 -x_* \Vert^2_\proj }{t^2}\, .
\end{align*}
		\item Assume further that $\mu > 0$ and choose the constant parameters $\eta_t = \eta_t' \equiv \sqrt{\frac{\mu}{L}}$, $\gamma_t \equiv \frac{1}{L}$, $\gamma_t' \equiv \frac{1}{\sqrt{\mu L }}$. Then , 
		\begin{align*}
		\E f(x_t) - f(x_*) \leq \left(f(x_0) - f(x_*) + \frac{\mu}{2} \Vert z_0 - x_* \Vert^2_\proj \right) \exp \left(-\sqrt{\frac{\mu}{L}}t\right) \, .
		\end{align*}	
	\end{enumerate}
\end{theorem}

\begin{proof}
Similarly to the proof in Appendix~\ref{ap:proof-additive}, the proof of this theorem is along the same lines as the proof of Theorem~\ref{thm:continuized}, and we only highlight the major differences. The process $\bar{x}_t = (t,x_t,z_t)$ satisfies the equation 
\begin{align*}
&\diff \bar{x}_t = b(\bar{x}_t) \diff t + \int_{\Xi}G(\bar{x}_t,\xi) \diff N(t,\xi), && b(\bar{x}_t) = \begin{pmatrix}
1 \\
\eta_t(z_t-x_t) \\
\eta_t'(x_t-z_t) 
\end{pmatrix}, &&G(\bar{x}_t,\xi) = \begin{pmatrix}
0 \\
-\gamma_t \frac{\proj_{\xi \xi}}{\cP_\xi} \nabla f(x_t, \xi) \\
-\gamma_t' \nabla f(x_t, \xi)
\end{pmatrix}.
\end{align*}
We also consider a slightly different Lyapunov function $\phi_t$ that takes into account the projector $\proj $:
\begin{equation*}
\phi_t = A_t\left(f(x_t)-f_*\right) + \frac{B_t}{2}\Vert z_t - x_* \Vert^2_\proj   
\end{equation*}
This change of norm is essential to take into account the fact that $f$ is not strongly convex with respect to the euclidean norm, but only with respect to $\Vert \cdot \Vert_\proj $. We apply Proposition \ref{prop:ito} to $\phi_t = \varphi(\bar{x}_t) = \varphi(t,x_t,z_t)$
and obtain
\begin{align}
\label{eq:decomposition-phi-additive}
\phi_t = \phi_0 + \int_{0}^{t} I_s \diff s + M_t \, ,
\end{align}
where $M_t$ is a martingale and
\begin{align*}
I_t = \langle \nabla \varphi(\bar{x}_t), b(\bar{x}_t) \rangle +\E_\xi\varphi(\bar{x}_{t}+G(\bar{x}_{t}, \xi) ) - \varphi(\bar{x}_{t}) \, . 
\end{align*}
The computation of the first term remains the same: the inequality \eqref{eq:aux-4}-\eqref{eq:aux-5} holds. The computation of the second term becomes 
\begin{align*}
\E_\xi \varphi(\bar{x}_{t}+G(\bar{x}_{t},\xi) ) - \varphi(\bar{x}_{t}) &= A_t\left(\E_\xi f\left(x_t-\gamma_t \frac{\proj_{\xi \xi}}{\cP_\xi} \nabla f(x_t, \xi)\right)-f(x_t)\right) \\
&\qquad+ \frac{B_t}{2} \left(\E_\xi\Vert (z_t-x_*) - \gamma_t' \nabla f(x_t, \xi) \Vert^2_\proj  - \Vert z_t-x_*\Vert^2_\proj \right) \, .
\end{align*}
As $f$ is $M$-smooth, 
\begin{align*}
f\left(x_t-\gamma_t\frac{\proj_{\xi \xi}}{\cP_\xi} \nabla f(x_t, \xi)\right)-f(x_t) &\leq \langle \nabla f(x_t) , - \gamma_t \frac{\proj_{\xi \xi}}{\cP_\xi} \nabla f(x_t, \xi) \rangle + \frac{1}{2} \Vert \gamma_t \frac{\proj_{\xi \xi}}{\cP_\xi} \nabla f(x_t, \xi) \Vert^2_M  \, .\end{align*}
In the additive case, the variance is bounded by $\sigma^2$. In this case, we have that:
\begin{equation}
    \Vert \frac{\proj_{\xi \xi}}{\cP_\xi} \nabla f(x_t, \xi) \Vert^2_M = \frac{M_{\xi \xi} \proj_{\xi \xi}}{\cP_\xi^2} \Vert \nabla f(x_t, \xi) \Vert^2_\proj  \leq L \Vert \nabla f(x_t, \xi)\Vert^2_\proj ,
\end{equation}
and similarly: 
\begin{equation}
    \langle \nabla f(x_t) , - \gamma_t \frac{\proj_{\xi \xi}}{\cP_\xi} \nabla f(x_t, \xi) \rangle = - \gamma_t \frac{\proj_{\xi \xi}}{\cP_\xi^2}\Vert \nabla_\xi f(x_t)\Vert^2 = \gamma_t \Vert \nabla f(x_t, \xi)\Vert^2_\proj .
\end{equation}
Thus:
\begin{align*}
\E_\xi f\left(x_t-\gamma_t \frac{\proj_{\xi \xi}}{\cP_\xi} \nabla f(x_t, \xi )\right)-f(x_t) 
&\leq \gamma_t (1 - \gamma_t L) \E_\xi \Vert \nabla f(x_t, \xi)\Vert^2_\proj \, .
\end{align*}
Similarly, thanks to the unbiasedness of $\nabla f(x_t, \xi)$, 
\begin{align*}
	\E_\xi\Vert (z_t-x_*) - \gamma_t' \nabla f(x_t,\xi) \Vert^2_\proj  &- \Vert z_t-x_*\Vert^2_\proj \\
	&= -2 \gamma_t' \langle \E_\xi \proj  \nabla f(x_t,\xi), z_t-x_* \rangle + \gamma_t'^2 \E_\xi \Vert  \nabla f(x_t,\xi) \Vert^2_\proj   \\
	&\leq -2 \gamma_t' \langle \nabla f(x_t), z_t-x_* \rangle + \gamma_t'^2 \E_\xi \Vert  \nabla f(x_t,\xi) \Vert^2_\proj  \, .
\end{align*}
This gives
\begin{align*}
\varphi(\bar{x}_{t}+G(\bar{x}_{t}) ) - \varphi(\bar{x}_{t}) \leq - & B_t \gamma_t' \langle \nabla f(x_t), z_t- x_* \rangle\\
&+ \left(B_t \gamma_t'^2 - A_t \gamma_t \left(2 - L\gamma_t\right)\right) \frac{1}{2} \E_\xi \Vert  \nabla f(x_t,\xi) \Vert^2_\proj   \, .
\end{align*}
Combining the bounds, we obtain 
\begin{align*}
I_t  &\leq \left(\frac{\diff A_t}{\diff t} - A_t \eta_t\right)  \langle \nabla f(x_t), x_t- x_* \rangle + \left(\frac{\diff B_t }{\diff t}- B_t\eta_t'\right) \frac{1}{2}\Vert z_t - x_* \Vert^2_\proj  \\
&\qquad+ (A_t \eta_t- B_t \gamma_t')  \langle \nabla f(x_t), z_t- x_* \rangle + \left(B_t \eta_t'- \frac{\diff A_t}{\diff t} \mu\right) \frac{1}{2}\Vert x_t - x_* \Vert^2_\proj \\ 
&\qquad  + \left(B_t \gamma_t'^2 - A_t \gamma_t \left(2 - L\gamma_t\right)\right) \frac{1}{2} \E_\xi \Vert  \nabla f(x_t,\xi) \Vert^2_\proj  \,.
\end{align*}
We see that we obtain a result that is very similar to that of the deterministic case.  The choices of parameters of Theorem \ref{thm:coord_descent} cancel all first five prefactors, and satisfy $\gamma_t = \frac{1}{L}$, $A_t L \gamma_t^2 = B_t \gamma_t'^2$. We thus obtain $I_t \leq 0$ and so $\phi_t$ is a supermartingale, and the rest follows as in Appendix~\ref{ap:proof-continuized}. 
\end{proof}

\section{Accelerated Decentralized Optimization with Randomized Gossip Communications.}
\label{app:decentralized_optim}

We now consider the setting of decentralized optimization considered in Section~\ref{sec:decentralized_opt}. More specifically, recall that we wish to solve:
\begin{equation}\label{eq:app:pbm_optim_decentralisé}
    \min_{x\in\R^d} \left\{ f(x)=\frac{1}{|\cV|}\sum_{v\in\cV} f_v(x)\right\},
\end{equation}
where the function $f_v$ is privately held by node $v \in \cV$. To solve this problem, a classical approach is to use a dual formulation~\citep{scaman2017optimal, hendrikx2018accelerated}. We first rewrite Problem~\eqref{eq:app:pbm_optim_decentralisé} as:
\begin{equation}\label{eq:app:pbm_optim_decentralisé_constrained}
    \min_{X\in\R^{|\cV|\times d}, \ X_u = X_v \ \forall \{u,v\} \in \cE } \left\{ F(X)=\frac{1}{|\cV|}\sum_{v\in\cV} f_v(X_v)\right\},
\end{equation}
where $X_v \in \R^d$ corresponds to the local parameter of node $v$, and the equality constraints ensures equivalence between~\eqref{eq:app:pbm_optim_decentralisé} and~\eqref{eq:app:pbm_optim_decentralisé_constrained}. The constraints are linear and can be expressed in matrix form as: 
\begin{equation}
    A^\top X = 0,
\end{equation}
with $A \in \R^{\cE \times \cV}$ such that $\ker(A^{\top})={\rm Span}(1,...,1)$ the constant vector. 
The natural choice for matrix $A$ is to choose a square root of the Laplacian matrix of graph $G$. For $(e_v)_{v\in\cV}$ and $(e_{\{v,w\}})_{\{v,w\}\in\E}$ the canonical bases of $\R^{\cV}$ and $\R^{\cE}$, $A$ is thus that for any $\{v,w\}\in\cE$:
\begin{equation*}
    Ae_{\{v,w\}}= \sqrt{\cP_{\{v,w\}}}(e_v-e_w).
\end{equation*}
Matrix $A$ then satisfies $AA^{\top}=\cL$ the Laplacian matrix of graph $G$ with weights $\cP_{\{v,w\}}$.
Indeed, if $W_\edgeuv = \cP_\edgeuv (e_v - e_w) (e_v - e_w)^\top$ corresponds to the gossip matrix for edge $\edgeuv$, $A$ is such that:
\begin{equation}
    A A^\top = \sum_{\edgeuv \in \cE} W_\edgeuv = \cL.
\end{equation} 
Then, introducing Lagrange multipliers $\lambda$, we obtain through Lagrangian duality that Problem~\eqref{eq:app:pbm_optim_decentralisé} is equivalent to:
\begin{equation}
    \label{eq:app:pb_dual}
    \max_{\lambda \in \R^{\cE \times d}} - F^*(A\lambda),
\end{equation}
with $F^*$ the convex conjugate of $F$. Following the approach of~\citet{hendrikx2018accelerated}, we then apply Accelerated Coordinate Descent to this dual problem. Yet, we use the \emph{continuized} version of Theorem~\ref{thm:coord_descent}, which allows us to remove the global iterations counter on which previous approaches rely. We see that Problem~\eqref{eq:app:pb_dual} has exactly the right form to apply Theorem~\ref{thm:coord_descent}, leading to the following dual iterations: 
\begin{align}
\begin{split}\label{eq:continuized_cd}
&\diff \lambdax_t = \eta_t (\lambdaz_t - \lambdax_t) \diff t - \gamma_t \int_{\R_{\ge0} \times \cE} \frac{R_\edgeuv}{\cP_\edgeuv^2} e_\edgeuv e_\edgeuv^\top A^\top \nabla F^*( A \lambdax_t) \diff N(t,\edgeuv) \, , \\
&\diff \lambdaz_t = \eta_t' (\lambdax_t - \lambdaz_t) \diff t- \gamma_t'  \int_{\R_{\ge0} \times \cE} \frac{1}{\cP_\edgeuv} e_\edgeuv e_\edgeuv^\top A^\top \nabla F^*( A \lambdax_t) \diff N(t,\edgeuv) \,,
\end{split}
\end{align}
where $P = A^\dagger A$ with $A^\dagger$ is the pseudo-inverse of $A$, $R_\edgeuv = e_\edgeuv^\top A^\dagger A e_\edgeuv$. Now, we multiply these iterations by $A$ on the left (which is standard), and we rewrite them with the following iterates:
\begin{equation}
    y_t = A \lambdax_t, \qquad z_t = A \lambdaz_t.
\end{equation}
Note that $y_t, z_t \in \R^{|\cV| \times d}$, and are thus variables associated with \emph{nodes} of the graph.  
\begin{align}
\begin{split}\label{eq:continuized_cd}
&\diff y_t = \eta_t (z_t - y_t) \diff t - \gamma_t \int_{\R_{\ge0} \times \cE} \frac{R_\edgeuv}{\cP_\edgeuv^2} W_\edgeuv \nabla F^*(y_t) \diff N(t,\edgeuv) \, , \\
&\diff z_t = \eta_t' (y_t - z_t) \diff t- \gamma_t'  \int_{\R_{\ge0} \times \cE} \frac{1}{\cP_\edgeuv} W_\edgeuv \nabla F^*( y_t) \diff N(t,\edgeuv) \,,
\end{split}
\end{align}
where we recall that $W_\edgeuv = \cP_\edgeuv (e_v - e_w) (e_v - e_w)^\top$ corresponds to the gossip matrix for edge $\edgeuv$. Besides, the dual gradients $\nabla F^*(y_t)$ are such that $e_v^\top \nabla F^*(y_t) = \nabla f_v^*( e_v^\top y_t)$, and so each component can be computed locally at node $v$. 

In summary, the distributed decentralized algorithm writes as follows.
Upon activation of edge $\{v_k,w_k\}$ at time $T_k$,
\begin{align}
\begin{split} \label{eq:accel_decentralized_optim}
    & G_\edgeuvk(T_k) = \omega_\edgeuvk \left[ \nabla f^*( (y_{T_k^-})_{v_k}) - \nabla f^*( (y_{T_k^-})_{w_k})\right]\\[10pt]
    & y_{T_k}(v_k)= y_{T_k^-}(v_k) - \gamma_t \frac{R_\edgeuvk}{\cP_\edgeuvk^2} G_\edgeuvk(T_k)\, , \\
    & y_{T_k}(w_k) = y_{T_k^-}(w_k) + \gamma_t \frac{R_\edgeuvk}{\cP_\edgeuvk^2} G_\edgeuvk(T_k)\, ,\\[10pt]
    &z_{T_k}(v_k)= z_{T_k^-}(v_k) - \gamma_t^\prime G_\edgeuvk(T_k)\, , \\
    &z_{T_k}(w_k) = z_{T_k^-}(w_k) + \gamma_t^\prime G_\edgeuvk(T_k)\, .
\end{split}
\end{align}
Between these updates, $y_t(v)$ and $z_t(v)$ locally mix at all nodes $v\in\cV$, according to the coupled ODE:
\begin{align*}
    \diff y_t(v)&= \eta_t (z_t(v)-y_t(v))\diff t,\\
    \diff z_t(v)&= \eta_t^\prime (y_t(v)-z_t(v))\diff t.
\end{align*}
This algorithm can be implemented with local computations and pairwise communications only, since an update along edge $\edgeuv$ only involves the parameters and functions of nodes $v$ and $w$. In order to fully describe this algorithm, we need to specify the various parameters. We do so, with the corresponding rate of convergence, in the following theorem.

\begin{theorem}[Asynchronous Accelerated Decentralized Optimization]
	\label{thm:decentralized_optim_appendix}
    Assume that each $f_v$ is $\mu$-strongly-convex with $\mu>0$ and $L$-smooth. Let $L_{\rm dual} = \frac{1}{\mu} \max_\edgeuv \frac{\proj_\edgeuv}{\cP_\edgeuv}$, where we recall that $R_\edgeuv = (A^\dagger A)_{\edgeuv, \edgeuv}$. Then, let $\theta^\prime_{\rm ARG} = \sqrt{\mu_{\rm gossip} / \max_\edgeuv \frac{\proj_\edgeuv}{\cP_\edgeuv}} $ where $\mu_{\rm gossip}$ is the smallest non-zero eigenvalue of the Laplacian of the graph $\cG$, and $\kappa = L / \mu$ is a bound on the condition number of $f$. We choose the constant parameters $\eta_t = \eta_t' \equiv \frac{\theta^\prime_{\rm ARG}}{\sqrt{\kappa}}$, $\gamma_t \equiv \frac{1}{L_{\rm dual}}$, $\gamma_t' \equiv \sqrt{\frac{L}{\mu_{\rm gossip} L_{\rm dual}}}$. The iterates produced by the algorithm described in~\eqref{eq:accel_decentralized_optim} verify:
		\begin{align*}
		\E \sum_{v\in V}\frac{1}{2}\|\nabla f_v^*(z_t(v)) - x_\star\|^2 \leq C_{0}^{\rm dual} \exp \left(-\frac{\theta^\prime_{\rm ARG}}{\sqrt{\kappa}}t\right) \,,
		\end{align*}	
with $C_0^{\rm dual} =  \frac{\lambda_{\max}(AA^\top)}{\mu}\left(F^*(A \lambdax_0) - F^*(A \lambda_\star) + \frac{\mu_{\rm gossip}}{2L} \Vert \lambdaz_0 - \lambda_\star \Vert^2_{A^\dagger A} \right)$, with $\lambda_\star$ a solution to the dual problem. 
\end{theorem}

Note that $\theta_{\rm ARG}^\prime$ is slightly different from $\theta_{\rm ARG}$. Yet, following~\citet{hendrikx2018accelerated}, an equivalent of Corollary~\ref{cor:accelerated_gossip} can be obtained for $\theta_{\rm ARG}^\prime$. To obtain Theorem~\ref{thm:acc_asy_decentralized_optim}, we simply choose $\lambdax_0 = \lambdaz_0$ and bound the dual function suboptimality by the distance to optim using the smoothness and strong convexity of $F^*$.

We stress the fact that the accelerated algorithm described in this section, as well as accelerated randomized gossip in Section \ref{sec:gossip}, are decentralized and asynchronous: operations are local and do not require any global synchronization, provided that a continuous time clock can be shared. This is possible only thanks to the continuized framework. However, there are some limitations: even if these algorithms are the first to achieve these rates without any global synchronization, computations and communications are here assumed to happen instantly, or to take a negligible time. Handling communication and computation physical capacity constraints such as delays or node/edge overloads in our algorithms as in \cite{even2021decentralized} combined with accelerated schemes is left for future works.

\begin{proof}
First note that the Hessian of the dual objective writes for some $\lambda \in \R^{|\cE| \times d}$:
\begin{equation}
    A^\top \nabla^2 F^*(A\lambda) A \succcurlyeq \frac{1}{L} A^\top A,
\end{equation}
since $F^*$ is $L^{-1}$ strongly-convex when $F$ is $L$-smooth~\citep{kakade2009duality}. Thus, the dual objective is $\mu_{\rm gossip} / L$ strongly convex on the orthogonal of the kernel of $A$. Similarly, the smoothness of the dual objective in direction $\edgeuv$ is equal to:
\begin{equation}
    M_{\edgeuv \edgeuv} = e_\edgeuv^\top A^\top \nabla^2 F^*(A\lambda) A e_\edgeuv \preccurlyeq \frac{1}{\mu} e_\edgeuv^\top A^\top A e_\edgeuv = \frac{\cP_\edgeuv}{2\mu}.
\end{equation}
Thus, we have that:
\begin{equation}
    L_{\rm dual} = \max_\edgeuv \frac{M_{\edgeuv\edgeuv} \proj_\edgeuv}{\cP_\edgeuv^2} = \frac{1}{\mu} \max_\edgeuv \frac{\proj_\edgeuv}{\cP_\edgeuv}.
\end{equation}
Then, the result follows directly from applying Theorem~\eqref{thm:coord_descent}, together with the smoothness of the dual gradients, since:
\begin{equation}
    \E \sum_{v\in V}\frac{1}{2}\Vert\nabla f_v^*(z_t(v)) - x_\star\Vert^2 \leq \E \frac{1}{2\mu} \Vert A \lambdaz_t - A \lambda_\star \Vert^2 \leq \frac{\lambda_{\max}(AA^\top)}{2\mu} \E\Vert \lambdaz_t - \lambda_\star \Vert^2_\proj.
\end{equation}
\end{proof}

Note that the primal parameter that we are interested in is $x_t = \nabla f^*(z_t)$, and not $y_t$ or $z_t$ which are dual parameters.

\end{document}